\documentclass[11pt,reqno]{amsart}

\usepackage[utf8]{inputenc}
\usepackage{amsfonts}
\usepackage{amsmath}
\usepackage{amssymb}
\usepackage{amsthm}
\usepackage{mathrsfs}
\usepackage{stmaryrd}
\usepackage{color}
\usepackage[english]{babel}
\usepackage{fontenc}
\usepackage{url}
\usepackage{graphicx}
\usepackage{hyperref}
\usepackage{caption}
\usepackage{epstopdf}
\usepackage{hyphenat}
\usepackage{float}
\usepackage{indentfirst}
\usepackage{url}
\usepackage[export]{adjustbox}
\usepackage{tikz}
\usetikzlibrary{matrix,arrows,decorations.pathmorphing}

\allowdisplaybreaks

\newtheorem{theorem}{Theorem}[section]
\newtheorem{lemma}[theorem]{Lemma}
\newtheorem{proposition}[theorem]{Proposition}
\newtheorem{corollary}[theorem]{Corollary}

\theoremstyle{definition}
\newtheorem{observation}[theorem]{Observation}
\newtheorem{remark}[theorem]{Remark}
\newtheorem{notation}[theorem]{Notation}
\newtheorem{example}[theorem]{Example}
\newtheorem{definition}[theorem]{Definition}

\DeclareMathOperator{\Sym}{Sym}
\DeclareMathOperator{\Supp}{Supp}
\DeclareMathOperator{\red}{red}
\DeclareMathOperator{\Bl}{Bl}
\DeclareMathOperator{\Spec}{Spec}
\DeclareMathOperator{\Pic}{Pic}
\DeclareMathOperator{\Aut}{Aut}
\DeclareMathOperator{\MP}{MP}
\DeclareMathOperator{\fr}{fr}
\DeclareMathOperator{\sm}{sm}
\DeclareMathOperator{\even}{even}
\DeclareMathOperator{\odd}{odd}
\DeclareMathOperator{\Gr}{Gr}
\DeclareMathOperator{\Proj}{Proj}

\DeclareMathOperator{\Hom}{Hom}
\DeclareMathOperator{\Cone}{Cone}

\makeatletter
\def\l@subsection{\@tocline{2}{0pt}{2.5pc}{5pc}{}}
\makeatother

\author{Luca Schaffler}
\address{Department of Mathematics, KTH Royal Institute of Technology, SE-100 44 Stockholm, Sweden}
\email{lucsch@math.kth.se}

\title{The KSBA compactification of the moduli space of $D_{1,6}$-polarized Enriques surfaces}

\begin{document}
\maketitle
\begin{abstract}
We describe a compactification by stable pairs (also known as KSBA compactification) of the $4$-dimensional family of Enriques surfaces which arise as the $\mathbb{Z}_2^2$-covers of the blow up of $\mathbb{P}^2$ at three general points branched along a configuration of three pairs of lines. Up to a finite group action, we show that this compactification is isomorphic to the toric variety associated to the secondary polytope of the unit cube. We relate the KSBA compactification considered to the Baily--Borel compactification of the same family of Enriques surfaces. Part of the KSBA boundary has a toroidal behavior, another part is isomorphic to the Baily--Borel compactification, and what remains is a mixture of these two. We relate the stable pair compactification studied here with Looijenga's semitoric compactifications.
\end{abstract}


\section{Introduction}
In the study of moduli spaces, it is important to provide compactifications which are functorial and with meaningful geometric and combinatorial properties. A leading example in this sense is the Deligne--Mumford and Knudsen compactification of the moduli space of smooth $n$-pointed curves. Another relevant example is Alexeev's compactification of the moduli space of principally polarized abelian varieties, which extended previous work of Mumford, Namikawa, and Nakamura \cite{Ale02,AN99,Mum72,Nam76}. Another case that was intensely studied is that of K3 surfaces, especially in degree $2$ see \cite{Sha80,Loo86,Sca87,Hac04,Laz16,AET19}. Similarly, compactifications of the moduli space of degree $2$ Enriques surfaces were studied by Sterk \cite{Ste91,Ste95}, who compared Shah's GIT compactification for degree $2$ Enriques surfaces in \cite{Sha81} with a semitoric Looijenga compactification \cite{Loo03}. In this paper, we study a special $4$-dimensional family of Enriques surfaces. By using the theory of stable pairs \cite{Ale94,Ale96,KSB88,Kol18}, we produce a geometric compactification which we describe explicitly and relate to other standard compactifications, such as the Baily--Borel.

Enriques surfaces were classically constructed as the normalization of the vanishing locus in $\mathbb{P}^3$ of the following equation:
\begin{equation*}
aW_0^2W_1^2W_2^2+bW_0^2W_1^2W_3^2+cW_0^2W_2^2W_3^2+dW_1^2W_2^2W_3^2+W_0W_1W_2W_3q(W_0,W_1,W_2,W_3)=0,
\end{equation*}
where $q$ is a non-degenerate quadratic form. We consider the case where $q$ is diagonal, which is an alternative description of the four dimensional family of Enriques surfaces studied in \cite{Oud10}. These arise as follows. Let $\Bl_3\mathbb{P}^2$ be the blow up of $\mathbb{P}^2$ at three general points. We have three distinct fibrations $\pi_i\colon\Bl_3\mathbb{P}^2\rightarrow\mathbb{P}^1$, $i=1,2,3$, and we choose two distinct irreducible rulings $\ell_i,\ell_i'$ for each fibration. Then an appropriate $\mathbb{Z}_2^2$-cover branched along $\sum_{i=1}^3(\ell_i+\ell_i')$ gives an Enriques surface (see Definition~\ref{definitiond16polarizedenriquessurface}). These Enriques surfaces occur in connection with Campedelli surfaces (see Remark~\ref{connectionwithcampedelli}), and Oudompheng described the Baily--Borel compactification of their period domain. In the current paper, we construct a geometric compactification of the moduli space $\mathbf{M}$ for such Enriques surfaces via Koll\'ar--Shepherd-Barron--Alexeev (KSBA) stable pairs.

Since Enriques surfaces are not of general type, we cannot use \cite{KSB88} directly, but we can do so by choosing a natural divisor transforming them into pairs of log general type: namely, we consider stable pairs $(S,\epsilon R)$, where $R$ is the ramification divisor of the above $\mathbb{Z}_2^2$-cover $S\rightarrow\Bl_3\mathbb{P}^2$ and $\epsilon$ is a small positive rational number. Now the KSBA machinery applies, enabling us to construct a compactification $\overline{\mathbf{M}}$ with geometric meaning. We have a complete description of the structure of this compactification and of the degenerate surfaces parametrized by the boundary.

\begin{theorem}[Theorem~\ref{combinatoricsstratificationstatement} and Corollary~\ref{generaldegenerationsofenriques}]
\label{firstmainresult}
The boundary of $\overline{\mathbf{M}}$ consists of two divisorial irreducible components and another irreducible component of codimension $3$. The surfaces parametrized by the general point of each one of these components are
\begin{enumerate}
\item the gluing of three del Pezzo surfaces of degree $2$ so that the dual complex is a $2$-simplex and the double locus consists of three smooth rational curves;
\item the gluing of two weak del Pezzo surfaces of degree $1$ along an elliptic curve;
\item the gluing of $\mathbb{P}^1\times\mathbb{P}^1$ and an elliptic ruled surface along a $(2,2)$ curve and a reduced fiber.
\end{enumerate}

For a full description of the stratification of the boundary of $\overline{\mathbf{M}}$ and the degenerations parametrized by it, see \S\ref{globalgeography} and \S\ref{studyofthecorrespondingz22cover}.
\end{theorem}

The first main tool in the proof of the above theorem is \cite{AP12}, which allows us to study degenerations of the pairs $\left(\Bl_3\mathbb{P}^2,\frac{1+\epsilon}{2}\sum_{i=1}^3(\ell_i+\ell_i')\right)$ instead. The corresponding degenerations of Enriques surfaces can be obtained after taking an appropriate $\mathbb{Z}_2^2$-cover. Similar ideas were used in \cite{AP09} for Campedelli surfaces, which are closely related to our example, and in \cite{AET19} for degree $2$ K3 surfaces. Now, the study of moduli compactifications of such pairs is related to the compactification of six lines in $\mathbb{P}^2$, for which the theory of hyperplane arrangements applies \cite{Ale15,HKT09}. However, as we work with $\Bl_3\mathbb{P}^2$, the stability condition in our situation is somewhat different, yielding a different compactification (see also Remark~\ref{comparingdegenofsixlines}). Therefore, instead of the theory of hyperplane arrangements, we use the theory of stable toric pairs in \cite{Ale02}, which in turn gives an explicit global description of the stable pair compactification $\overline{\mathbf{M}}$ as we will soon see. More precisely, denote by $\Delta$ the toric boundary of $(\mathbb{P}^1)^3$ and let $B\subseteq(\mathbb{P}^1)^3$ be a general effective divisor of class $(1,1,1)$. Note that $B$ is isomorphic to $\Bl_3\mathbb{P}^2$ and that $\Delta|_B$ consists of six lines. If $Q$ denotes the unit cube, then $((\mathbb{P}^1)^3,B)$ is a stable toric pair of type $Q$ according to Definition~\ref{definitionofstabletoricpair}. Let $\overline{\mathbf{M}}_Q$ be the coarse moduli space parametrizing the stable toric pairs $((\mathbb{P}^1)^3,B)$ and their degenerations. We prove the following.

\begin{theorem}[Theorem~\ref{propertiesofthemapstptoksba}]
\label{secondmaintheorem}
Let $\Sym(Q)$ be the symmetry group of $Q$. Then $\overline{\mathbf{M}}_Q/\Sym(Q)\cong\overline{\mathbf{M}}$. On a dense open subset, the isomorphism maps the $\Sym(Q)$-class of a stable toric pair $(X,B)$ to the appropriate $\mathbb{Z}_2^2$-cover of $\left(B,\left(\frac{1+\epsilon}{2}\right)\Delta|_B\right)$, where $\Delta$ denotes the toric boundary of $X$.
\end{theorem}

The main difficulty in the proof of Theorem~\ref{secondmaintheorem} is that, away from the mentioned dense open subset, $\left(B,\left(\frac{1+\epsilon}{2}\right)\Delta|_B\right)$ is not stable. This happens if and only if the polyhedral subdivision of $Q$ associated to $(X,B)$ has what we call a \emph{corner cut} (see Definition~\ref{definitionofcornercut}). This is analyzed in \S\ref{modificationsofthepairs}, where we describe a modification of $(X,B)$, denoted by $(X^\bullet,B^\bullet)$, such that $\left(B^\bullet,\left(\frac{1+\epsilon}{2}\right)\Delta^\bullet|_{B^\bullet}\right)$ is a stable pair. As $\overline{\mathbf{M}}_Q$ is the projective toric variety associated to the secondary polytope of $Q$, the isomorphism in Theorem~\ref{secondmaintheorem} yields an explicit description of $\overline{\mathbf{M}}$.

A different compactification can be constructed by Hodge theory. As for K3 surfaces, the Enriques surfaces considered here can be characterized Hodge-theoretically as being $D_{1,6}$-polarized in the sense of Dolgachev \cite{Dol96}. The moduli space for Enriques surfaces, as for K3 surfaces, is the quotient of a bounded Hermitian symmetric domain $\mathcal{D}$ of type IV by the action of an appropriate arithmetic group $\Gamma$. In this case, there is a natural compactification of $\mathcal{D}/\Gamma$, namely the Baily--Borel compactification $\overline{\mathcal{D}/\Gamma}^*$. As previously mentioned, \cite{Oud10} studied this particular situation and described the boundary of $\overline{\mathcal{D}/\Gamma}^*$. Furthermore, Oudompheng showed in \cite[\S4]{Oud10} that $\overline{\mathcal{D}/\Gamma}^*$ is isomorphic to the quotient by a finite group of the GIT compactification of the moduli space of six lines in $\mathbb{P}^2$. The next theorem relates the KSBA compactification $\overline{\mathbf{M}}$ with $\overline{\mathcal{D}/\Gamma}^*$.

\begin{theorem}[Theorem~\ref{morphismtobailyborelcompactification}]
There exists a birational morphism $\overline{\mathbf{M}}\rightarrow\overline{\mathcal{D}/\Gamma}^*$ extending the period map to the boundary of $\overline{\mathbf{M}}$.
\end{theorem}

Let us take a closer look to the morphism $\overline{\mathbf{M}}\rightarrow\overline{\mathcal{D}/\Gamma}^*$. Oudompheng showed that the boundary of $\overline{\mathcal{D}/\Gamma}^*$ consists of three $0$-dimensional boundary components, called $0$-cusps, and two $1$-dimensional boundary components, called $1$-cusps. With reference to Figure~\ref{pictureboundarybailyborel}, we observe that the boundary of the KSBA compactification $\overline{\mathbf{M}}$ has a toroidal behavior in a neighborhood of the preimage of the even $0$-cusp, and is isomorphic to the Baily--Borel compactification in a neighborhood of the preimage of the odd $0$-cusp of type $2$. Above the middle odd $0$-cusp of type $1$ the behavior of $\overline{\mathbf{M}}$ is not toroidal or Baily--Borel. These considerations make us consider Looijenga's semitoric compactifications (see \cite{Loo03}), which generalize the Baily--Borel and toroidal compactifications in the case of type IV Hermitian symmetric domains. A semitoric compactification $(\overline{\mathcal{D}/\Gamma})_\Sigma$ depends on the choice of $\Sigma$, which is an admissible decomposition of the conical locus (see \S\ref{admissibledecompositionoftheconicallocusofd}). To construct $\Sigma$ in our case, for each $0$-cusp we consider the associated hyperbolic lattice, and we consider the subdivision of one connected component of $x^2>0$ given by the mirrors of the reflections with respect to the vectors of square $-1$. To study the stratification of $(\overline{\mathcal{D}/\Gamma})_\Sigma$ we compute a fundamental domain for the discrete reflection group generated by the reflections with respect to the $(-1)$-vectors. This calculation is not included in the current paper, but it can be found in \cite[\S8.6]{Sch17}. The next theorem gives a set-theoretic comparison of the boundaries of $\overline{\mathbf{M}}$ and $(\overline{\mathcal{D}/\Gamma})_\Sigma$.

\begin{theorem}[Theorem~\ref{bijectionstrataksbasemitoriccompactification}]
The admissible decomposition $\Sigma$ in Definition~\ref{ourchoiceofadmissibledecomposition} produces a semitoric compactification $(\overline{\mathcal{D}/\Gamma})_\Sigma$ birational to $\overline{\mathbf{M}}$ and whose boundary strata are in bijection with the boundary strata of $\overline{\mathbf{M}}$. This bijection preserves the dimensions of the strata and the intersections between them.
\end{theorem}

We expect this to be an isomorphism, but we can not yet prove it. This will be object of future investigation. We mention that in \cite{AET19} a geometric compactifications of moduli of K3 surfaces is identified with a semitoric compactification which is not toroidal or Baily--Borel.

The paper is organized as follows. In \S\ref{enriquesandchoicedivisor} we define the Enriques surfaces of interest and the moduli space we want to compactify. In \S\ref{preliminaries-stable-pairs-stable-toric-pairs} we briefly recall the theory of stable pairs and stable toric pairs that we need. The technical results in \S\ref{modificationsofthepairs} and \S\ref{wallcrossingstudy} allow us to construct the morphism $\overline{\mathbf{M}}_Q\rightarrow\overline{\mathbf{M}}$ in \S\ref{sectiononksbacompactification}, where we also describe the stable pairs parametrized by the boundary of $\overline{\mathbf{M}}$ and its stratification. In \S\ref{relationwithbailyborel} we construct the morphism $\overline{\mathbf{M}}\rightarrow\overline{\mathcal{D}/\Gamma}^*$. Finally, the connection with Looijenga's semitoric compactifications is discussed in \S\ref{looijengasemitoriccompactifications}. We work over $\mathbb{C}$.


\section*{Acknowledgments}
I would like to express my gratitude to my advisor, Valery Alexeev, for guiding me through this project and for his help. I am also grateful to Robert Varley for many discussions related to this work. Many thanks to Patricio Gallardo for great suggestions and to Adrian Brunyate for helpful conversations. I would like to thank Renzo Cavalieri, Klaus Hulek, Dino Lorenzini, Rita Pardini, Roberto Svaldi, Giancarlo Urz\'ua, and Alessandro Verra for their feedback. I also thank the anonymous referees for many valuable suggestions. The figures in this paper were realized using the software GeoGebra, Copyright \copyright~International GeoGebra Institute, 2013. I gratefully acknowledge financial support from the Office of the Graduate School of the University of Georgia, the NSF grant DMS-1603604, and the Research and Training Group in Algebra, Algebraic Geometry, and Number Theory, at the University of Georgia.


\section{$D_{1,6}$-polarized Enriques surfaces}
\label{enriquesandchoicedivisor}


\subsection{Enriques surfaces and $\mathbb{Z}_2^2$-covers}
\label{d16polarizedenriquessurfaces}

A \emph{variety} is a connected and reduced scheme of finite type over $\mathbb{C}$ (in particular, a variety need not be irreducible). A \emph{surface} is a $2$-dimensional projective variety. An \emph{Enriques surface} $Y$ is a smooth irreducible surface with $2K_Y\sim0$ and $h^0(Y,\omega_Y)=h^1(Y,\mathcal{O}_Y)$. These properties are enough to imply that $S$ is minimal with Kodaira dimension $0$, $h^0(Y,\omega_Y)=0$, and $h^1(Y,\mathcal{O}_Y)=0$.

\begin{definition}
\label{definitiond16polarizedenriquessurface}
Let $\Bl_3\mathbb{P}^2$ be the blow up of $\mathbb{P}^2$ at $[1:0:0],[0:1:0],[0:0:1]$. Then $\Bl_3\mathbb{P}^2$ comes with three genus zero pencils $\pi_1,\pi_2,\pi_3\colon\Bl_3\mathbb{P}^2\rightarrow\mathbb{P}^1$. Denote by $\ell_i,\ell_i'$ two distinct irreducible elements in the $i$-th pencil, $i=1,2,3$. Assume that the divisor $\sum_{i=1}^3(\ell_i+\ell_i')$ has no triple intersection points. In what follows, we use the general theory of abelian covers developed in \cite{Par91}.

Let $\mathbb{Z}_2^2=\{e,a,b,c\}$, where $e$ is the identity element, and let $\{\chi_0,\chi_1,\chi_2,\chi_3\}$ be the characters of $\mathbb{Z}_2^2$ with $\chi_0=1$ and $\chi_1(b)=\chi_1(c)=\chi_2(a)=\chi_2(c)=\chi_3(a)=\chi_3(b)=-1$. Define $D_a=\ell_1+\ell_1',D_b=\ell_2+\ell_2',D_c=\ell_3+\ell_3'$. Consider the building data (see \cite[Definition 2.1]{Par91}) consisting of the divisors $D_a,D_b,D_c$ and the line bundles $L_{\chi_1},L_{\chi_2},L_{\chi_3}$ satisfying
\begin{equation*}
2L_{\chi_1}=D_b+D_c,~2L_{\chi_2}=D_a+D_c,~2L_{\chi_3}=D_a+D_b.
\end{equation*}
This building data determines a $\mathbb{Z}_2^2$-cover $\pi\colon S\rightarrow\Bl_2\mathbb{P}^2$ branched along $\sum_{i=1}^3(\ell_i+\ell_i')$, which is unique up to isomorphism of $\mathbb{Z}_2^2$-covers (see \cite[Theorem 2.1]{Par91}). By \cite[Proposition 3.1]{Par91} we have that $S$ is smooth, and using \cite[Proposition 4.2, formula (4.8)]{Par91} one can compute that $\chi(\mathcal{O}_S)=1$, which implies that $h^0(S,\omega_S)=h^1(S,\mathcal{O}_S)$. If $R$ denotes the ramification divisor of the cover, then $K_S\sim\pi^*(K_{\Bl_3\mathbb{P}^2})+R$ and $2R\sim\pi^*(\sum_{i=1}^3(\ell_i+\ell_i'))$ imply that $2K_S\sim0$, hence $S$ is an Enriques surface. These are the Enriques surfaces studied in \cite{Oud10}, and to make a clear connection with Oudompheng's paper we also call them \emph{$D_{1,6}$-polarized Enriques surfaces}. The reason for this name is explained in Remark~\ref{explainthename}. Observe that the ramification divisor of the $\mathbb{Z}_2^2$-cover consists of six genus one curves.
\end{definition}

\begin{remark}
$D_{1,6}$-polarized Enriques surfaces can be described in the following alternative way. In $\mathbb{P}^2$, consider six distinct lines $\overline{\ell}_i,\overline{\ell}_i'$, $i=1,2,3$, so that the divisor $\sum_{i=1}^3(\overline{\ell}_i+\overline{\ell}_i')$ does not have triple intersection points. Let $S'\rightarrow\mathbb{P}^2$ be the $\mathbb{Z}_2^2$-cover with building data analogous to the one in Definition~\ref{definitiond16polarizedenriquessurface}. Note that $S'$ is singular, and its singular locus consists of exactly six $A_1$ singularities: two above each intersection point $\overline{\ell}_i\cap\overline{\ell}_i'$. The minimal resolution $S\rightarrow S'$ is a $D_{1,6}$-polarized Enriques surface.

Given this alternative definition, it is natural to ask what is the connection between the Enriques surface $S$ above and the K3 surface $Z$ given by the minimal resolution of the double cover of $\mathbb{P}^2$ branched along $\sum_{i=1}^3(\overline{\ell}_i+\overline{\ell}_i')$. This can be understood by looking at the universal K3 cover $X\rightarrow S$. The K3 surface $X$ can be viewed as the minimal resolution of an appropriate $\mathbb{Z}_2^3$-cover of $\mathbb{P}^2$ branched along the six lines, and $Z$ can be obtained as the minimal resolution of the quotient of $X$ by an appropriate order four subgroup of $\mathbb{Z}_2^3$. It turns out that the N\'eron--Severi lattice of a very general K3 surface $X$ is not isometric to the N\'eron--Severi lattice of a very general K3 surface $Z$, so these two families of K3 surfaces are distinct. All this is studied in detail in \cite{Sch18}.
\end{remark}

\begin{remark}
\label{explainthename}
Let $D_{1,6}$ denote the index $2$ sublattice of $\langle1\rangle\oplus\langle-1\rangle^{\oplus6}$ of vectors of even square. Note that $D_{1,6}$ is isometric to $U\oplus D_5$. Let $e_0,e_1,\ldots,e_6$ be the canonical basis of $\langle1\rangle\oplus\langle-1\rangle^{\oplus6}$. According to \cite{Oud10}, a $D_{1,6}$-polarized Enriques surface $S$ can be equivalently defined to be an Enriques surface whose Picard group contains a primitively embedded copy of $D_{1,6}$ such that:
\begin{enumerate}
\item The vector $2e_0$ corresponds to a nef divisor class $H$ (which is the preimage of a general line in $\Bl_3\mathbb{P}^2$ under the $\mathbb{Z}_2^2$-cover);
\item Let $C_1,C_2,C_3$ be the exceptional divisors of the blow up $\Bl_3\mathbb{P}^2\rightarrow\mathbb{P}^2$. The preimage of $C_i$ under the $\mathbb{Z}_{2}^2$-cover consists of two disjoint smooth rational curves that we denote by $R_i^+,R_i^-$. Then we ask for the vectors $e_1\pm e_2,e_3\pm e_4,e_5\pm e_6$ to correspond to the six irreducible curves $R_1^\pm,R_2^\pm,R_3^\pm$ respectively.
\end{enumerate}
Note that for $i=1,2,3$, the linear system $|H-R_i^+-R_i^-|$ is a genus one pencil, and the preimages in $S$ of the lines $\ell_i,\ell_i'$ in Definition~\ref{definitiond16polarizedenriquessurface} give the two half-fibers of this pencil (see \cite[Chapter~VIII, \S17]{BHPV04} for the definition of half-fiber).
\end{remark}

\begin{remark}
\label{connectionwithcampedelli}
If $S$ is a $D_{1,6}$-polarized Enriques surface, then the divisor $C=H+\sum_{i=1}^3(R_i^++R_i^-)$ is divisible by $2$ in $\Pic(S)$. The $\mathbb{Z}_2$-cover of $S$ branched along $C$ has six $(-1)$-curves. Blowing down these curves we obtain a Campedelli surface with (topological) fundamental group $\mathbb{Z}_2^3$ (these were considered in \cite{AP09}). Conversely, such a Campedelli surface $X$ can be realized as the $\mathbb{Z}_2^3$-cover of $\mathbb{P}^2$ branched along seven lines. The minimal desingularization of the quotient of $X$ by the involution fixing pointwise the preimage of one of these lines is a $D_{1,6}$-polarized Enriques surface.
\end{remark}


\subsection{The family of $D_{1,6}$-polarized Enriques surfaces}
\label{sectiononthefamilyofd16polarizedenriquessurfaces}

\begin{definition}
\label{familyD16polarizedEnriquessurfaces}
Let $((c_{000},c_{100},\ldots,c_{111}),([X_0X_1],[Y_0,Y_1],[Z_0,Z_1]))$ be coordinates in $\mathbb{G}_m^8\times(\mathbb{P}^1)^3$. Let $\mathcal{X}''\subseteq\mathbb{G}_m^8\times(\mathbb{P}^1)^3$ be the closed subscheme defined by the vanishing of
\[
\sum_{i,j,k=0,1}c_{ijk}X_i^2Y_j^2Z_k^2=0.
\]
Let $\mathbf{U}\subseteq\mathbb{G}_m^8$ be the dense open subset such that the corresponding fibers in $\mathcal{X}''$ are smooth. Let $\mathcal{X}'=\mathcal{X}''|_\mathbf{U}$. If $X\subseteq\mathcal{X}'$ is a fiber, then $X$ is a smooth hypersurfaces of class (2,2,2) in $(\mathbb{P}^1)^3$. Therefore, $X$ is a K3 surface ($K_X\sim0$ by the adjunction formula and $h^1(X,\mathcal{O}_X)=0$ can be computed using the long exact sequence in cohomology associated to $0\rightarrow\mathscr{I}_X\rightarrow\mathcal{O}_{(\mathbb{P}^1)^3}\rightarrow\mathcal{O}_X\rightarrow0$). Moreover, $X$ comes with a fixed-point-free involution given by
\[
\iota\colon([X_0:X_1],[Y_0:Y_1],[Z_0:Z_1])\mapsto([X_0:-X_1],[Y_0:-Y_1],[Z_0:-Z_1]).
\]
Hence, $X/\iota$ is an Enriques surface (see \cite[Proposition VIII.17]{Bea96}). Let us show $X/\iota$ is a $D_{1,6}$-polarized Enriques surface. The restriction to $X$ of the morphism
\[
([X_0:X_1],[Y_0:Y_1],[Z_0:Z_1])\mapsto([X_0^2:X_1^2],[Y_0^2:Y_1^2],[Z_0^2:Z_1^2])
\]
realizes $X$ as a $\mathbb{Z}_2^3$-cover of $B\subseteq(\mathbb{P}^1)^3$ given by
\[
\sum_{i,j,k=0,1}c_{ijk}X_iY_jZ_k=0.
\]
We have that $B$ is a del Pezzo surface of degree $6$ (hence, $B\cong\Bl_3\mathbb{P}^2$), and the branch locus of $X\rightarrow B$ is given by $\Delta|_B$, where $\Delta$ is the toric boundary of $(\mathbb{P}^1)^3$. Observe that $\Delta|_B$ consists of six lines, two for each genus zero fibration $B\rightarrow\mathbb{P}^1$, without triple intersection points. So $X/\iota\rightarrow B$ is the $\mathbb{Z}_2^2$-cover that gives a $D_{1,6}$-polarized Enriques surface. Notice that there are several ways to take the $\mathbb{Z}_2^2$-cover of $B$ branched along $\Delta|_B$ by appropriately varying the building data. However, these other choices produce rational surfaces or K3 surfaces. The involution $\iota$ acts on the whole family $\mathcal{X}'$, so that $\mathcal{X}=\mathcal{X}'/\iota\rightarrow\mathbf{U}$ is a family of $D_{1,6}$-polarized Enriques surfaces.
\end{definition}

\begin{remark}
\label{descriptionusingenriquessextics}
The general Enriques surface with degree $6$ polarization can be realized as the normalization of the vanishing locus in $\mathbb{P}^3$ of the following equation:
\begin{equation*}
aW_0^2W_1^2W_2^2+bW_0^2W_1^2W_3^2+cW_0^2W_2^2W_3^2+dW_1^2W_2^2W_3^2+W_0W_1W_2W_3q(W_0,W_1,W_2,W_3)=0,
\end{equation*}
where $q$ is a non-degenerate quadratic form (see \cite[\S4]{Muk12}). By \cite[Proposition 4.1]{Muk12}, the universal K3 cover of such Enriques surfaces is an appropriate $(2,2,2)$ hypersurface in $(\mathbb{P}^1)^3$ invariant under the involution $\iota$ above. Under this correspondence, the K3 surfaces in Definition~\ref{familyD16polarizedEnriquessurfaces} are the universal covers of the degree $6$ Enriques surfaces for which the quadratic form $q$ is diagonal, providing an alternative description of the $D_{1,6}$-polarized Enriques surfaces. More explicitly, consider the sextic hypersurface $S$ given by
\begin{align*}
&c_{001}W_0^2W_1^2W_2^2+c_{010}W_0^2W_1^2W_3^2+c_{100}W_0^2W_2^2W_3^2+c_{111}W_1^2W_2^2W_3^2\\
+&W_0W_1W_2W_3(c_{000}W_0^2+c_{011}W_1^2+c_{101}W_2^2+c_{110}W_3^2)=0.
\end{align*}
Let $X\subseteq(\mathbb{P}^1)^3$ be the K3 surface given by $\sum_{i,j,k=0,1}c_{ijk}X_i^2Y_j^2Z_k^2=0$. The morphism $X\rightarrow S$ is explicitly given by the restriction to $X$ of
\begin{align*}
(\mathbb{P}^1)^3&\rightarrow\mathbb{P}^3,\\
([X_0:X_1],[Y_0:Y_1],[Z_0:Z_1])&\mapsto[X_0Y_0Z_0:X_0Y_1Z_1:X_1Y_0Z_1:X_1Y_1Z_0].
\end{align*}
\end{remark}

The next proposition shows that the family $\mathcal{X}\rightarrow\mathbf{U}$ in Definition~\ref{familyD16polarizedEnriquessurfaces} captures all possible $D_{1,6}$-polarized Enriques surfaces, and illustrates which fibers are pairwise isomorphic.

\begin{proposition}
\label{mainconstruction}
Consider $\Bl_3\mathbb{P}^2$ together with a divisor $\sum_{i=1}^3(\ell_i+\ell_i')$ (see Definition~\ref{definitiond16polarizedenriquessurface}) without triple intersection points. Then there exists $B=V\left(\sum_{i,j,k=0,1}c_{ijk}X_iY_jZ_k\right)\subseteq(\mathbb{P}^1)^3$ with coefficients $c_{ijk}\neq0$ such that $\left(\Bl_3\mathbb{P}^2,\sum_{i=1}^3(\ell_i+\ell_i')\right)$ is isomorphic to $(B,\Delta|_B)$. Moreover, such $B\subseteq(\mathbb{P}^1)^3$ is uniquely determined up to the action of $\mathbb{G}_m^4\rtimes\Sym(Q)$ on the coefficients of $B$, where $Q$ is the unit cube. Hence, $\mathbf{U}/(\mathbb{G}_m^4\rtimes\Sym(Q))$ is the moduli space of $D_{1,6}$-polarized Enriques surfaces with our choice of divisor.
\end{proposition}

\begin{proof}
Consider the three projections $\pi_i\colon\Bl_3\mathbb{P}^2\rightarrow\mathbb{P}^1$, $i=1,2,3$. Let $\ell_i=\pi_i^{-1}([a_{0i}:a_{1i}])$ and $\ell_i'=\pi_i^{-1}([a_{0i}':a_{1i}'])$. The morphism $(\pi_1,\pi_2,\pi_3)\colon\Bl_3\mathbb{P}^2\rightarrow(\mathbb{P}^1)^3$ is an embedding whose image is a divisor of class $(1,1,1)$ (observe that the restriction of $\Delta$ to this divisor gives the six $(-1)$-curves, each one with multiplicity $2$). For each one of the three copies of $\mathbb{P}^1$ choose an automorphism $\varphi_i$, $i=1,2,3$, sending $[a_{0i}:a_{1i}],[a_{0i}':a_{1i}']$ to $[1:0],[0:1]$ respectively. Let $B$ be the image of the composition of the embedding $(\pi_1,\pi_2,\pi_3)$ followed by the automorphism $(\varphi_1,\varphi_2,\varphi_3)$. Then, under this morphism, $\Bl_3\mathbb{P}^2$ is isomorphic to $B$ and $\sum_{i=1}^3(\ell_i+\ell_i')$ corresponds to $\Delta|_B$. Moreover, $B=V\left(\sum_{i,j,k=0,1}c_{ijk}X_iY_jZ_k\right)$, where the coefficients $c_{ijk}$ are nonzero, or otherwise $\Delta|_B$ would have triple intersection points.

In the construction of $B$ above we made some choices. The group $\Sym(Q)$ acts by permuting the three projections $\pi_1,\pi_2,\pi_3$, and for each $i$ it exchanges $[a_{0i}:a_{1i}]$ and $[a_{0i}':a_{1i}']$ (note that $\Sym(Q)$ is isomorphic to the wreath product $\mathbb{Z}_2\wr S_3$). Each $\varphi_i$ is uniquely determined up to $\mathbb{G}_m$, and an additional $\mathbb{G}_m$ acts by rescaling the coefficients $c_{ijk}$. This describes an action of $\mathbb{G}_m^4\rtimes\Sym(Q)$ on the vector of coefficients $(c_{000},c_{100},\ldots,c_{111})$. Observe that our construction of $B$ is invariant under the action of $\Aut(\Bl_3\mathbb{P}^2)$ on the line arrangement (see \cite[Theorem 8.4.2]{Dol12} for the description of $\Aut(\Bl_3\mathbb{P}^2)$).

To conclude, we show that $B=V\left(\sum_{i,j,k=0,1}c_{ijk}X_iY_jZ_k\right)\subseteq(\mathbb{P}^1)^3$ with $c_{ijk}\neq0$ such that $\left(\Bl_3\mathbb{P}^2,\sum_{i=1}^3(\ell_i+\ell_i')\right)\cong(B,\Delta|_B)$ is unique up to the above $(\mathbb{G}_m^4\rtimes\Sym(Q))$-action. But this is true because, up to $\mathbb{G}_m^4\rtimes\Sym(Q)$, there is a unique way to realize $\Bl_3\mathbb{P}^2$ in $(\mathbb{P}^1)^3$ so that the six $(-1)$-curves are given by the restriction of $\Delta$, and this is given by $V(X_0Y_0Z_0-X_1Y_1Z_1)$ (we omit the proof of this).
\end{proof}

We now construct a compactification of $\mathbf{U}/(\mathbb{G}_m^4\rtimes\Sym(Q))$ using stable pairs.


\section{Preliminaries: moduli of stable pairs and stable toric pairs}
\label{preliminaries-stable-pairs-stable-toric-pairs}

\subsection{Stable pairs}

Our main reference is \cite{Kol13}.

\begin{definition}
Let $X$ be a variety and let $B=\sum_ib_iB_i$ be a divisor on $X$ where $b_i\in(0,1]\cap\mathbb{Q}$ and $B_i$ is a prime divisor for all $i$. Then the pair $(X,B)$ is \emph{semi-log canonical} if the following conditions are satisfied:
\begin{enumerate}
\item $X$ is $S_2$ and every codimension $1$ point is regular or a double crossing singularity (varieties satisfying these properties are also called \emph{demi-normal});
\item If $\nu\colon X^\nu\rightarrow X$ is the normalization with conductors $D\subseteq X$ and $D^\nu\subseteq X^\nu$ (see \cite[\S5.1]{Kol13}), then the support of $B$ does not contain any irreducible component of $D$;
\item $K_X+B$ is $\mathbb{Q}$-Cartier;
\item The pair $(X^\nu,D^\nu+\nu_*^{-1}B)$, where $\nu_*^{-1}B$ denotes the strict transform of $B$ under $\nu$, is \emph{log canonical} (see \cite[Definition 2.8]{Kol13}).
\end{enumerate}
\end{definition}

\begin{definition}
A pair $(X,B)$ is \emph{stable} if the following conditions are satisfied:
\begin{enumerate}
\item \emph{Singularities:} $(X,B)$ is a semi-log canonical pair;
\item \emph{Numerical:} $K_X+B$ is ample.
\end{enumerate}
\end{definition}

The following is our main example of interest.

\begin{lemma}
\label{reductiontobaseofthecover}
Let $S$ be a $D_{1,6}$-polarized Enriques surface and let $\pi\colon S\rightarrow\Bl_3\mathbb{P}^2$ be the corresponding $\mathbb{Z}_2^2$-cover ramified at $E=\sum_{i=1}^3(E_i+E_i')$ and branched along $L=\sum_{i=1}^3(\ell_i+\ell_i')$ (see Definition~\ref{definitiond16polarizedenriquessurface}). Then
\begin{equation*}
K_S+\epsilon E\sim_\mathbb{Q}\pi^*\left(K_{\Bl_3\mathbb{P}^2}+\left(\frac{1+\epsilon}{2}\right)L\right).
\end{equation*}
In particular, $(S,\epsilon E)$ is a stable pair if and only if $\left(\Bl_3\mathbb{P}^2,\left(\frac{1+\epsilon}{2}\right)L\right)$ is a stable pair.
\end{lemma}

\begin{proof}
We have that $K_S\sim\pi^*(K_{\Bl_3\mathbb{P}^2})+E$ and $2E\sim\pi^*(L)$. This implies that $K_S\sim_\mathbb{Q}\pi^*(K_{\Bl_3(\mathbb{P}^2)})+\frac{1}{2}\pi^*(L)$, and by adding $\epsilon E\sim_\mathbb{Q}\frac{\epsilon}{2}\pi^*(L)$ to both sides we obtain what claimed. For the last statement about stability, we have that $(S,\epsilon E)$ is semi-log canonical if and only if $\left(\Bl_3\mathbb{P}^2,\left(\frac{1+\epsilon}{2}\right)L\right)$ is semi-log canonical by \cite[Lemma 2.3]{AP12}, and $K_S+\epsilon E$ is ample if and only if $K_{\Bl_3\mathbb{P}^2}+\left(\frac{1+\epsilon}{2}\right)L$ is ample as $\pi$ is finite and surjective.
\end{proof}


\subsection{Moduli of stable pairs}
\label{formulationviehwegmodulistack}

In what follows we compactify $\mathbf{U}/(\mathbb{G}_m^4\rtimes\Sym(Q))$ by taking its Zariski closure inside an appropriate projective moduli space of stable pairs. Our main references are \cite{Ale15,Kol18}.

\begin{definition}
The \emph{Viehweg's moduli stack} $\overline{\mathcal{M}}_{d,N,C,\underline{b}}$ is defined as follows. Let us fix constants $d,N\in\mathbb{Z}_{>0}$, $C\in\mathbb{Q}_{>0}$, and $\underline{b}=(b_1,\ldots,b_n)$ with $b_i\in(0,1]\cap\mathbb{Q}$ and $Nb_i\in\mathbb{Z}$ for all $i=1,\ldots,n$. For any reduced scheme $S$ over $\mathbb{C}$, $\overline{\mathcal{M}}_{d,N,C,\underline{b}}(S)$ is the set of proper flat families $\mathfrak{X}\rightarrow S$ together with a divisor $\mathfrak{B}=\sum_ib_i\mathfrak{B}_i$ satisfying the following properties:
\begin{enumerate}
\item For all $i=1,\ldots,n$, $\mathfrak{B}_i$ is a codimension one closed subscheme such that $\mathfrak{B}_i\rightarrow S$ is flat at the generic points of $\mathfrak{X}_s\cap\Supp{\mathfrak{B}_i}$ for every $s\in S$;
\item Every geometric fiber $(X,B)$ is a stable pair of dimension $d$ with $(K_X+B)^d=C$;
\item There exists an invertible sheaf $\mathscr{L}$ on $\mathcal{X}$ such that for every geometric fiber $(X,B)$ one has $\mathscr{L}|_X\cong\mathcal{O}_X(N(K_X+B))$.
\end{enumerate}
\end{definition}

\begin{definition}
Consider the moduli stack $\overline{\mathcal{M}}_{6\epsilon^2}=\overline{\mathcal{M}}_{d,N,C,\underline{b}}$ with $d=2$, $\underline{b}=(b_1,b_2,b_3)=\left(\frac{1+\epsilon}{2},\frac{1+\epsilon}{2},\frac{1+\epsilon}{2}\right)$ (because we want three pairs of divisors, and we do not distinguish divisors in the same pair) where $0<\epsilon\ll1$ is a fixed rational number and $C=6\epsilon^2$. The reason for this coefficient is because if $\left(\Bl_3\mathbb{P}^2,\left(\frac{1+\epsilon}{2}\right)L\right)$ is as in Proposition~\ref{mainconstruction}, then $\left(K_{\Bl_3\mathbb{P}^2}+\left(\frac{1+\epsilon}{2}\right)L\right)^2=6\epsilon^2$. For a suitably chosen positive integer $N$ depending on $d,\underline{b}$ and $C$ (which does not need to be specified, see \cite[3.13]{Ale96}), the Viehweg's moduli functor $\overline{\mathcal{M}}_{6\epsilon^2}$ is coarsely represented by a projective scheme, which we denote by $\overline{\mathbf{M}}_{6\epsilon^2}$. This is true because we are working with surface pairs ($d=2$) and our coefficients are strictly greater than $\frac{1}{2}$ (see \cite[Theorem 1.6.1, case 2]{Ale15}).
\end{definition}

\begin{observation}
\label{s3actiononksbaspace}
The group $S_3$ has a natural action on the Viehweg moduli stack $\overline{\mathcal{M}}_{6\epsilon^2}$ given by permuting the labels of the three divisors $\mathfrak{B}_1,\mathfrak{B}_2,\mathfrak{B}_3$. In particular, we have an induced $S_3$-action on the coarse moduli space $\overline{\mathbf{M}}_{6\epsilon^2}$.
\end{observation}

\begin{definition}
With reference to Definition~\ref{familyD16polarizedEnriquessurfaces}, let $\mathcal{B}'$ be the closed subscheme of $\mathbf{U}\times(\mathbb{P}^1)^3$ given by the vanishing of the equation
\[
\sum_{i,j,k=0,1}c_{i,j,k}X_iY_jZ_k=0.
\]
Let $\mathcal{L}'$ be the restriction to $\mathcal{B}'$ of the toric boundary of $\mathbb{G}_m^8\times(\mathbb{P}^1)^3$. Then the family of stable pairs $\left(\mathcal{B}',\left(\frac{1+\epsilon}{2}\right)\mathcal{L}'\right)\rightarrow\mathbf{U}$ descends to a family of stable pairs $\left(\mathcal{B},\left(\frac{1+\epsilon}{2}\right)\mathcal{L}\right)\rightarrow\mathbf{U}/(\mathbb{G}_m^4\rtimes\Sym(Q))$.

Then there is an induced morphism $f\colon\mathbf{U}/(\mathbb{G}_m^4\rtimes\Sym(Q))\rightarrow\overline{\mathbf{M}}_{6\epsilon^2}/S_3$ which is injective by Proposition~\ref{mainconstruction}. We define $\overline{\mathbf{M}}$ to be the normalization of the closure of the image of $f$. We refer to $\overline{\mathbf{M}}$ as the \emph{KSBA compactification} as the moduli space of $D_{1,6}$-polarized Enriques surfaces.
\end{definition}

\begin{remark}
\label{equivalenttoconsiderthebaseofthecover}
With the notation introduced in Lemma~\ref{reductiontobaseofthecover}, we claim that studying the degenerations of the stable pairs $(S,\epsilon E)$ is equivalent to studying of the degenerations of the stable pairs $\left(\Bl_3\mathbb{P}^2,\left(\frac{1+\epsilon}{2}\right)L\right)$. To prove this, let $K$ be the field of fractions of a DVR $(A,\mathfrak{m})$, where $\mathfrak{m}$ is the maximal ideal of $A$. Let $(\mathfrak{S}^\circ,\mathcal{E}^\circ)$ (resp. $(\mathfrak{B}^\circ,\mathcal{L}^\circ)$) be a family of stable pairs over $\Spec(K)$ with fibers isomorphic to $(S,\epsilon E)$ (resp. $\left(\Bl_3\mathbb{P}^2,\left(\frac{1+\epsilon}{2}\right)L\right)$). Let $\mathfrak{S}^\circ\rightarrow\mathfrak{B}^\circ$ be the appropriate $\mathbb{Z}_2^2$-cover ramified at $\mathcal{E}^\circ$ and branched along $\mathcal{L}^\circ$.

Let $(\mathfrak{S},\mathcal{E})$ be the completion of $(\mathfrak{S}^\circ,\mathcal{E}^\circ)$ to a family of stable pairs over $\Spec(A)$, or a ramified base change of it (see \cite[Theorem 2.1]{Ale06}). Denote by $(\mathfrak{S}_\mathfrak{m},\mathcal{E}_\mathfrak{m})$ the central fiber of $(\mathfrak{S},\mathcal{E})$. Similarly, define $(\mathfrak{B},\mathcal{L})$ and $(\mathfrak{B}_\mathfrak{m},\mathcal{L}_\mathfrak{m})$ for $(\mathfrak{B}^\circ,\mathcal{L}^\circ)$. We want to show that $(\mathfrak{S}_\mathfrak{m},\mathcal{E}_\mathfrak{m})$ is a $\mathbb{Z}_2^2$-cover of $(\mathfrak{B}_\mathfrak{m},\mathcal{L}_\mathfrak{m})$. This is automatic if we can show that the $\mathbb{Z}_2^2$-action on $\mathfrak{S}^\circ$ extends to $\mathfrak{S}$, because the quotient of $\mathfrak{S}$ by $\mathbb{Z}_2^2$ is isomorphic to $\mathfrak{B}$ by the uniqueness of the completion of $\mathfrak{B}^\circ$ over $\Spec(A)$.

Fix any $g\in\mathbb{Z}_2^2$ and consider the corresponding action $\alpha_g\colon\mathfrak{S}\dashrightarrow\mathfrak{S}$. If we resolve the indeterminacies

\begin{center}
\begin{tikzpicture}[>=angle 90]
\matrix(a)[matrix of math nodes,
row sep=2em, column sep=2em,
text height=1.5ex, text depth=0.25ex]
{&\mathfrak{S}'&\\
\mathfrak{S}&&\mathfrak{S},\\};
\path[->] (a-1-2) edge node[right]{}(a-2-1);
\path[->] (a-1-2) edge node[above right]{$\alpha_g'$}(a-2-3);
\path[dashed,->] (a-2-1) edge node[below]{$\alpha_g$}(a-2-3);
\end{tikzpicture}
\end{center}
then $\alpha_g'$ corresponds to a morphism $\overline{\alpha}_g$ from the log canonical model of $(\mathfrak{S}',\mathcal{E}')$ to $\mathfrak{S}$ (\cite[Definition 1.19]{Kol13}). But the log canonical model of $(\mathfrak{S}',\mathcal{E}')$ is $(\mathfrak{S},\mathcal{E})$, so $\overline{\alpha}_g\colon\mathfrak{S}\rightarrow\mathfrak{S}$ is the desired extension of $\alpha_g$.
\end{remark}

To study $\overline{\mathbf{M}}$ we apply techniques from \cite{Ale02}, which we now recall.


\subsection{Stable toric pairs}
\label{stabletoricpairs}
Fix a torus $T$ over $\mathbb{C}$ and let $M$ be its character lattice. Let $M_\mathbb{R}$ denote the tensor product $M\otimes_\mathbb{Z}\mathbb{R}$. Although our main interest is over $\mathbb{C}$, note that the theory of stable toric pairs we are about to review works in any characteristic.

Let $X$ be a variety with $T$-action. We say that $X$ is a \emph{stable toric variety} if $X$ is seminormal and its irreducible components are toric varieties under the $T$-action. The \emph{toric boundary} of a stable toric variety is defined to be the sum of the boundary divisors of each irreducible component which are not in common with other irreducible components. If $X$ is projective and $\mathscr{L}$ is an ample and $T$-linearized invertible sheaf on $X$, we say that the pair $(X,\mathscr{L})$ is a \emph{polarized} stable toric variety.

Assume that we have a polarized stable toric variety $(X,\mathscr{L})$. Then we can associate to each irreducible component $X_i$ of $X$ a lattice polytope $P_i$. These polytopes can be glued together in the same way as $X$ is the union of its irreducible components. This results into a topological space $\cup_iP_i$ which is called the \emph{topological type} of $X$. The topological type comes together with a finite map $\cup_iP_i\rightarrow M_\mathbb{R}$, called the \emph{reference map}, which embeds each $P_i$ as a lattice polytope in $M_\mathbb{R}$. The set of faces of the polytopes $P_i$, together with the identifications coming from the gluing, is a \emph{complex of polytopes}. Up to isomorphism, there is a $1$-to-$1$ correspondence between polarized stable toric varieties (for a fixed torus $T$) and the following data:
\begin{enumerate}
\item A complex of polytopes $\mathcal{P}$;
\item A reference map $\cup_{P\in\mathcal{P}}P\rightarrow M_\mathbb{R}$;
\item An element of the cohomology group we are about to describe. For each $P\in\mathcal{P}$, let $C_P$ be the saturated sublattice of $\mathbb{Z}\oplus M$ generated by $(1,P)$, and let $T_P$ be the torus $\Hom(C_P,\mathbb{C}^*)$. The tori $T_P$ for $P\in\mathcal{P}$ define a sheaf of abelian groups $\underline{T}$ on the poset $\mathcal{P}$ with the order topology. An element of $H^1(\mathcal{P},\underline{T})$ describes the way the irreducible components of $X$ are glued together.
\end{enumerate}
For more details see \cite[\S4.3]{Ale06} and \cite[Theorem 1.2.6]{Ale02}.

\begin{definition}
Let $(X,\mathscr{L})$ be a polarized stable toric variety and let $Q\subseteq M_\mathbb{R}$ be a lattice polytope. We say that $X$ has \emph{type $Q$} if the complex of polytopes $\mathcal{P}$ associated to $X$ is a polyhedral subdivision of the marked polytope $(Q,Q\cap M)$. In this case the reference map is given by the inclusion of $Q$ in $M_\mathbb{R}$. Furthermore, the toric boundary of $X$ is the sum of the divisors corresponding to the facets in $\mathcal{P}$ contained in the boundary of $Q$.
\end{definition}

\begin{definition}
\label{definitionofstabletoricpair}
A \emph{stable toric pair} is a polarized stable toric variety $(X,\mathscr{L})$ together with an effective Cartier divisor $B$ which is the divisor of zeros of a global section of $\mathscr{L}$. Also, we require that $B$ does not contain any torus fixed point (or equivalently any $T$-orbit). We denote such stable toric pair simply by $(X,B)$ as $\mathscr{L}\cong\mathcal{O}_X(B)$. Two stable toric pairs $(X,B)$ and $(X',B')$ are isomorphic if there exists an isomorphism $f\colon X\rightarrow X'$ that preserves the $T$-action and such that $f^*B'=B$. We say that a stable toric pair has \emph{type} $Q$ if the corresponding polarized stable toric variety has type $Q$.
\end{definition}

\begin{remark}
\label{combinatorialdescriptionstabletoricpair}
Let $(X,B)$ be a stable toric pair and let $\mathcal{P}$ be the complex of polytopes corresponding to the polarized stable toric variety $(X,\mathcal{O}_X(B))$. If $X_i$ is an irreducible component of $X$, then the restriction $B|_{X_i}$ can be described combinatorially as follows (see \cite[Theorem 1.2.7]{Ale02} in combination with \cite[Lemma 2.2.7, part 2]{Ale02}). Consider the marking on the corresponding lattice polytope $P_i$ given by $P_i\cap M$. The lattice points in $P_i$ correspond to monomials as $M$ is the character lattice of the torus $T$. Now, $B|_{X_i}$ is determined (not uniquely) by a function $f\colon P_i\cap M\rightarrow\mathbb{C}^*$, which assigns to each monomial a corresponding coefficient, which we want to be nonzero because $B$ does not contain torus fixed points.
\end{remark}

\begin{example}
\label{stabletoricpairofinterest}
Let $B\subseteq(\mathbb{P}^1)^3$ be a smooth divisor of class $(1,1,1)$ not containing torus fixed points. Then $((\mathbb{P}^1)^3,B)$ is a stable toric pair of type $Q$, where $Q$ denotes the unit cube. Recall that an appropriate $\mathbb{Z}_2^2$-cover of $B$ branched along $\Delta|_B$ gives a $D_{1,6}$-polarized Enriques surface.
\end{example}


\subsection{The stack of stable toric pairs}

Given a locally noetherian scheme $S$ over $\mathbb{C}$, a \emph{family of stable toric pairs} $(\mathfrak{X},\mathfrak{B})\rightarrow S$ is a proper flat morphism of schemes $\pi\colon\mathfrak{X}\rightarrow S$ together with a compatible $T_S=(T\times_\mathbb{C}S)$-action on $\mathfrak{X}$ and an effective Cartier divisor $\mathfrak{B}\subseteq\mathfrak{X}$ such that $\pi|_\mathfrak{B}$ is flat and the fiber $(\mathfrak{X}_s,\mathfrak{B}_s)$ over every geometric point $s\rightarrow S$ is a stable toric pair with the action induced by $T_S$. Two families of stable toric pairs over the same base are \emph{isomorphic} if there exists an isomorphism of pairs over $S$ preserving the torus action. Given a lattice polytope $Q$, we say that a family of stable toric pairs has \emph{type $Q$} if every geometric fiber has type $Q$.

\begin{remark}
\label{algebraicdatafamilyofstabletoricpairs}
Let $(\mathfrak{X},\mathfrak{B})\rightarrow S$ be a family of stable toric pairs. Denote by $\pi$ the morphism $\mathfrak{X}\rightarrow S$ and let $\mathscr{L}=\mathcal{O}_\mathfrak{X}(\mathfrak{B})$. Following \cite[Proof of Lemma 2.10.1]{Ale02}, for $d\geq0$ the sheaves $\pi_*\mathscr{L}^{\otimes d}$ are locally free, and the torus action gives a decomposition $\pi_*\mathscr{L}^{\otimes d}=\bigoplus_{m\in M}R_{(d,m)}$ into sheaves that are also locally free. This results into a locally free $(\mathbb{Z}\oplus M)$-graded $\mathcal{O}_S$-algebra $R=\bigoplus_{(d,m)\in\mathbb{Z}\oplus M}R_{(d,m)}$ together with a section $\theta$ of $R$ of $\mathbb{Z}$-degree $1$ corresponding to $\mathfrak{B}$. Conversely, a pair $(R,\theta)$, where $R$ is a locally free graded $\mathcal{O}_S$-algebra and $\theta$ a degree $1$ section, uniquely determines a family of stable toric pairs up to isomorphism (see \cite[Proof of Theorem 2.10.8]{Ale02} and \cite[\S4]{Ale06}).
\end{remark}

Given a lattice polytope $Q$, we have a notion of stack $\mathscr{M}_Q$ over $\mathbb{C}$, whose objects are families of stable toric pairs of type $Q$. The following theorem is due to Alexeev.

\begin{theorem}[{\cite[Theorem 1.2.15]{Ale02}}]
\label{modulistabletoricpairs}
Let $Q$ be a lattice polytope and let $\mathscr{M}_Q$ be the stack of stable toric pairs of type $Q$. Then the following hold:
\begin{enumerate}
\item $\mathscr{M}_Q$ is a proper Deligne--Mumford stack with finite stabilizers;
\item $\mathscr{M}_Q$ has a coarse moduli space $\overline{\mathbf{M}}_Q'$ which is a projective scheme;
\item The boundary of $\overline{\mathbf{M}}_{Q}'$ is stratified according to the polyhedral subdivision of $(Q,Q\cap M)$;
\item The normalization $\overline{\mathbf{M}}_Q$ of the main irreducible component of $(\overline{\mathbf{M}}_Q')_{\red}$ is isomorphic to the toric variety associated to the secondary polytope $\Sigma(Q\cap M)$.
\end{enumerate}
\end{theorem}

\begin{remark}
\label{irreducibilitycoarsemoduliofstabletoricpairs}
Let $Q$ be the unit cube. The moduli space of stable toric pairs $\overline{\mathbf{M}}_Q$ parametrizes in a dense open subset the stable toric pairs $((\mathbb{P}^1)^3,B)$ in Example~\ref{stabletoricpairofinterest}. In this particular case, the moduli space $\overline{\mathbf{M}}_Q$ is irriducible because, if $\mathcal{P}$ is any polyhedral subdivisions of $(Q,Q\cap\mathbb{Z}^3)$, then $\mathcal{P}$ is regular and $H^1(\mathcal{P},\underline{T})=\{1\}$ (this cohomology group was introduced earlier in \S\ref{stabletoricpairs}). To prove the former, we know from \cite[Corollary 2.9]{San01} that a marked polytope with a nonregular subdivision has a nonregular triangulation. But all the triangulations of $(Q,Q\cap\mathbb{Z}^3)$ are regular by \cite[Theorem 3.2]{DL96}. For the proof that $H^1(\mathcal{P},\underline{T})=\{1\}$, we refer to \cite[\S6.3]{Sch17}.

Since $\overline{\mathbf{M}}_Q$ is the toric variety associated to the secondary polytope $\Sigma(Q\cap\mathbb{Z}^3)$, we also know that $\dim(\overline{\mathbf{M}}_Q)=\#(Q\cap\mathbb{Z}^3)-\dim(Q)-1=4$ \cite[Chapter 7, Theorem 1.7]{GKZ08}.
\end{remark}

\begin{proposition}
\label{rationalmapfromstptoksba}
Let $Q$ be the unit cube. Then there exists a rational map $\overline{\mathbf{M}}_Q\dashrightarrow\overline{\mathbf{M}}$ whose fibers are $\Sym(Q)$-orbits.
\end{proposition}

\begin{proof}
Let $\mathcal{P}$ be the complex of polytopes given by $Q$ and its faces. Let $C$ be the set of vertices of $Q$. Following \cite[Definition 2.6.6]{Ale02}, let $\MP^{\fr}[\mathcal{P},C](\mathbb{C})$ be the groupoid of stable toric pairs $((\mathbb{P}^1)^3,B)$ over $\mathbb{C}$ with the linearized line bundle $\mathcal{O}_{(\mathbb{P}^1)^3}(B)$ in which the arrows are the isomorphisms identical on the torus $T$. By \cite[Lemma 2.6.7]{Ale02}, we have that $\MP^{\fr}[\mathcal{P},C](\mathbb{C})$ is equivalent to the quotient stack $[\mathbb{G}_m^8/\mathbb{G}_m^4]$, where $\mathbb{G}_m^8$ represents the space of coefficients of the divisor $B=V\left(\sum_{i,j,k=0,1}c_{ijk}X_iY_jZ_k\right)$ on $(\mathbb{P}^1)^3$ and $(\lambda,\mu_1,\mu_2,\mu_3)\in\mathbb{G}_m^4$ acts on $\mathbb{G}_m^8$ as follows:
\begin{equation*}
(\lambda,\mu_1,\mu_2,\mu_3)\cdot(\ldots,c_{ijk},\ldots)=(\ldots,\lambda\mu_1^i\mu_2^j\mu_3^kc_{ijk},\ldots).
\end{equation*}
Observe that the quotient $\mathbb{G}_m^8/\mathbb{G}_m^4$ exists as a scheme and it is isomorphic to $\mathbb{G}_m^4$. The stabilizers of the action $\mathbb{G}_m^4\curvearrowright\mathbb{G}_m^8$ are trivial. It follows that the quotient stack $[\mathbb{G}_m^8/\mathbb{G}_m^4]$ is finely represented by $\mathbb{G}_m^8/\mathbb{G}_m^4$. This gives a dense open subset $\mathbf{A}$ of $\overline{\mathbf{M}}_Q$, namely the maximal subtorus $\mathbb{G}_m^4$, with a universal family which can be realized as the quotient of $(\mathbb{G}_m^8\times(\mathbb{P}^1)^3,V(\sum_{i,j,k=0,1}c_{ijk}X_iY_jZ_k))$ by the $\mathbb{G}_m^4$-action on the coefficients $c_{ijk}$. Therefore, we have an induced morphism $\overline{\mathbf{M}}_Q\dashrightarrow\overline{\mathbf{M}}$ defined on $\mathbf{A}$.
\end{proof}


\section{Stable replacement in one-parameter families}
\label{modificationsofthepairs}

To realize the stable pair compactification $\overline{\mathbf{M}}$ as a finite quotient of the toric variety $\overline{\mathbf{M}}_Q$, we first want to show that the rational map $\varphi\colon\overline{\mathbf{M}}_Q\dashrightarrow\overline{\mathbf{M}}$ from Proposition~\ref{rationalmapfromstptoksba} is defined over $\overline{\mathbf{M}}_Q$. To achieve this we use \cite[Theorem~7.3]{GG14}: let $x\in\overline{\mathbf{M}}_Q$ and let $(A,\mathfrak{m})$ be a DVR with field of fractions $K$. Let $g\colon\Spec(K)\rightarrow\overline{\mathbf{M}}_Q$ with image in the torus of $\overline{\mathbf{M}}_Q$ such that, if $\overline{g}\colon\Spec(A)\rightarrow\overline{\mathbf{M}}_Q$ is the unique extension of $g$ by the valuative criterion of properness, then $x=\overline{g}(\mathfrak{m})$. Likewise, $\varphi\circ g$ has a unique extension $\overline{\varphi\circ g}$ as an $A$-point of $\overline{\mathbf{M}}$. We want to show that $\overline{\varphi\circ g}(\mathfrak{m})$ only depends on $x$, and that it is independent from the choice of $g$. To prove this, we need (1) an explicit description of one-parameter families of stable toric pairs, and (2) to compute the stable pair parametrized by $\overline{\varphi\circ g}(\mathfrak{m})\in\overline{\mathbf{M}}$. Part (1) is contained in \cite{Ale02}, and we recall it in \S\ref{familiesofstabletoricpairsoveradvr}. The work in the current section and \S\ref{wallcrossingstudy} addresses (2), which is not automatic from knowing the limiting stable toric pair. We will need a modification: namely, remove what we call \emph{corner cuts}, see \S\ref{subsection-on-corner-cuts}.


\subsection{Families of stable toric pairs $(\mathfrak{X},\mathfrak{B})$ over a DVR}
\label{familiesofstabletoricpairsoveradvr}
Let $(A,\mathfrak{m})$ be a DVR with field of fractions $K$, uniformizing parameter $t$, and residue field our fixed base field $\mathbb{C}$. Let $T$ be a torus over $\mathbb{C}$ with character lattice $M$. Let $Q\subseteq M_\mathbb{R}$ be a lattice polytope. Define
\begin{equation*}
\theta=\sum_{m\in Q\cap M}c_m(t)t^{h(m)}x^{(1,m)},
\end{equation*}
where, for all $m\in Q\cap M$, $c_m(t)\in A$, $c_m(0)\in\mathbb{C}^*$, and $h(m)\in\mathbb{Z}$. Observe that the map $m\mapsto h(m)$ gives a height function $h\colon Q\cap M\rightarrow\mathbb{Z}$. Let $Q^+\subseteq M_\mathbb{R}\oplus\mathbb{R}$ be the convex hull of the half-lines $(m,h(m)+\mathbb{R}_{\geq0})$, $m\in Q\cap M$, and let $\Cone(Q^+)\subseteq\mathbb{R}\oplus M_\mathbb{R}\oplus\mathbb{R}$ be the cone over $(1,Q^+)$ with vertex at the origin. Then $h$ defines the following $(\mathbb{Z}\oplus M)$-graded $A$-algebra:
\begin{equation*}
R=A[t^\ell x^{(d,m)}\mid(d,m,\ell)\in\Cone(Q^+)\cap(\mathbb{Z}\oplus M\oplus\mathbb{Z})].
\end{equation*}
Observe that $\theta\in R$ is an element of $\mathbb{Z}$-degree $1$. Let $(\mathfrak{X},\mathfrak{B})\rightarrow\Spec(A)$ be the family of stable toric pairs associated to $(R,\theta)$ (see Remark~\ref{algebraicdatafamilyofstabletoricpairs}). If $\eta$ is the generic point of $\Spec(A)$, then $\mathfrak{X}_\eta=Y\times\Spec(K)$, where $Y$ is the toric variety associated to the polytope $Q$. The central fiber $(\mathfrak{X}_\mathfrak{m},\mathfrak{B}_\mathfrak{m})$ is a stable toric pair whose corresponding polyhedral subdivision of $(Q,Q\cap M)$ is induced by the height function $h$, and hence it is a regular subdivision. The equation of $\mathfrak{B}_\mathfrak{m}$ is given by
\begin{equation*}
\sum_{m\in Q\cap M}c_m(0)x^{(1,m)}=0.
\end{equation*}
For more details about this construction we refer to \cite[\S2.8]{Ale02}.


\subsection{Corner cuts}
\label{subsection-on-corner-cuts}
\begin{definition}
\label{definitionofcornercut}
Let $Q$ be the unit cube. We call \emph{corner cut} a subpolytope of $Q$ which is equal to the convex hull of the points $(0,0,0),(1,0,0),(0,1,0),(0,0,1)$ up to a symmetry of $Q$ (see Figure~\ref{modificationsubdivision} on the left). We call \emph{apex} the vertex of the corner cut which is at the intersection of three edges of the cube.
\end{definition}

\begin{notation}
Let $\mathcal{P}$ be a polyhedral subdivision of a lattice polytope $Q$. We denote by $\mathcal{P}_i$ the set of $i$-dimensional faces in $\mathcal{P}$.
\end{notation}

\begin{definition}
\label{polyhedralsubdivisionpdot}
Let $\mathcal{P}$ be a polyhedral subdivision of $(Q,Q\cap\mathbb{Z}^3)$. We define a polyhedral subdivision $\mathcal{P}^\bullet$ of $(Q,Q\cap\mathbb{Z}^3)$ via the following algorithm:
\begin{itemize}
\item[(1)] $\mathcal{R}=\mathcal{P}$;
\item[(2)] If $\mathcal{R}$ contains no corner cut, define $\mathcal{P}^\bullet=\mathcal{R}$ and stop. Otherwise, go to step (3);
\item[(3)] Let $P\in\mathcal{R}$ be a corner cut and let $R\in\mathcal{R}$ be that unique polytope sharing exactly a facet with $P$. Define $\mathcal{S}=(\mathcal{R}_3\setminus\{P,R\})\cup\{P\cup R\}$. Then redefine $\mathcal{R}$ to be the polyhedral subdivision of $Q$ generated by $\mathcal{S}$. Go to step (2).
\end{itemize}
In Figure~\ref{modificationsubdivision} we give an explicit example of $\mathcal{P}^\bullet$ given $\mathcal{P}$.
\end{definition}

\begin{figure}
\centering
\includegraphics[scale=0.45,valign=t]{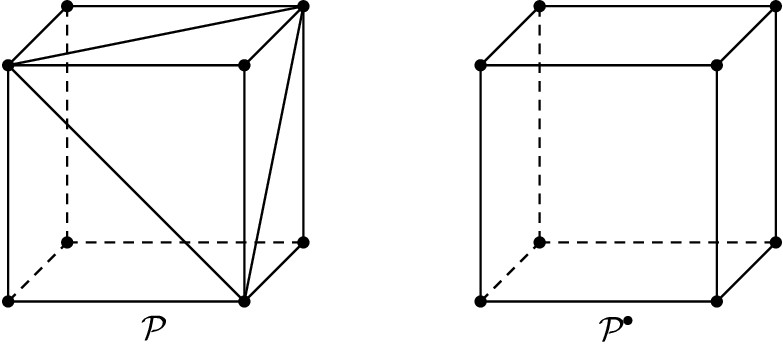}
\caption{Modification $\mathcal{P}^\bullet$ of the polyhedral subdivision $\mathcal{P}$.}
\label{modificationsubdivision}
\end{figure}


\subsection{The modified family $(\mathfrak{X}^\bullet,\mathfrak{B}^\bullet)$}
\label{themodifiedfamilywithbullet}

Let $Q$ be the unit cube and let $(\mathfrak{X},\mathfrak{B})$ be a family of stable toric pairs of type $Q$ over $\Spec(A)$. Define $X=\mathfrak{X}_\mathfrak{m}$, $B=\mathfrak{B}_\mathfrak{m}$, and assume that the fiber of $\mathfrak{X}$ over the generic point is isomorphic to $(\mathbb{P}^1)_K^3$, where recall $K$ is the field of fractions of $A$. Observe that the pair $\left(B,\left(\frac{1+\epsilon}{2}\right)\Delta|_B\right)$ may be not stable: first of all, the restriction $\Delta|_B$ can be defined if $\Delta$ is $\mathbb{Q}$-Cartier (see Proposition~\ref{toricboundarycartier}). Therefore, we want to define a new family $(\mathfrak{X}^\bullet,\mathfrak{B}^\bullet)$ which is isomorphic to the original one in the complement of the central fiber, and such that the new central fiber $(X^\bullet,B^\bullet)$ satisfies that $\left(B^\bullet,\left(\frac{1+\epsilon}{2}\right)\Delta^\bullet|_{B^\bullet}\right)$ is a stable pair, where $\Delta^\bullet$ denotes the toric boundary of $X^\bullet$. The stability of the pair $\left(B^\bullet,\left(\frac{1+\epsilon}{2}\right)\Delta^\bullet|_{B^\bullet}\right)$ is proved in \S\ref{wallcrossingstudy}.

To construct $(\mathfrak{X}^\bullet,\mathfrak{B}^\bullet)$ we define another $(\mathbb{Z}\oplus M)$-graded $A$-algebra induced by $\theta$ as follows. Denote by $\mathcal{P}$ the regular polyhedral subdivision of $(Q,Q\cap\mathbb{Z}^3)$ associated to $(\mathfrak{X}_\mathfrak{m},\mathfrak{B}_\mathfrak{m})$. Assume that $P\in\mathcal{P}_3$ is a corner cut and let $P'\in\mathcal{P}$ be that unique polytope sharing exactly one facet with $P$. Denote by $L$ the unique hyperplane in $\mathbb{R}^3\oplus\mathbb{R}$ containing the points $(m,h(m))$ for $m\in P'\cap\mathbb{Z}^3$. If $m$ is the apex of the corner cut $P$, then there exists a unique positive rational number $q_m$ such that $(m,h(m)-q_m)\in L$. Moreover, up to a finite ramified base change, we can assume that $q_m$ is integral. Let us consider the height function
\begin{displaymath}
h^\bullet(m)=\left\{ \begin{array}{ll}
h(m)-q_m&\textrm{if $m$ is the apex of a corner cut},\\
h(m)&\textrm{otherwise}.
\end{array} \right.
\end{displaymath}
Define a new $(\mathbb{Z}\oplus M)$-graded $A$-algebra $R^\bullet$ as we did in \S\ref{familiesofstabletoricpairsoveradvr}, but using $h^\bullet$ in place of $h$. Observe that $R\subseteq R^\bullet$ is a degree preserving embedding of graded algebras. Therefore $\theta\in R^\bullet$ is an element of $\mathbb{Z}$-degree $1$ and the pair $(R^\bullet,\theta)$ corresponds to a family $\mathfrak{X}^\bullet\rightarrow\Spec(A)$ of stable toric varieties together with a Cartier divisor $\mathfrak{B}^\bullet\subseteq\mathfrak{X}^\bullet$ given by the vanishing of $\theta$. Observe that $(\mathfrak{X}_\eta,\mathfrak{B}_\eta)$ and $((\mathfrak{X}^\bullet)_\eta,(\mathfrak{B}^\bullet)_\eta)$ are isomorphic over $\Spec(K)$ by construction. The equation of $(\mathfrak{B}^\bullet)_\mathfrak{m}$ is given by
\begin{equation*}
\sum_{\substack{m\in Q\cap M,\\m~\textrm{is not an apex}}}c_m(0)x^{(1,m)}=0.
\end{equation*}
So, in other words, we set equal to zero the coefficients of a monomials if it correspond to the apex of a corner cut in $\mathcal{P}$. For this reason, the central fiber $((\mathfrak{X}^\bullet)_\mathfrak{m},(\mathfrak{B}^\bullet)_\mathfrak{m})$ is not a stable toric pair if and only if $\mathcal{P}$ contains a corner cut. On the other hand, $((\mathfrak{X}^\bullet)_\mathfrak{m},\mathcal{O}((\mathfrak{B}^\bullet)_\mathfrak{m}))$ is always a polarized stable toric variety whose corresponding polyhedral subdivision of $(Q,Q\cap\mathbb{Z}^3)$ is $\mathcal{P}^\bullet$ (see Definition~\ref{polyhedralsubdivisionpdot}). Finally, observe that if $\mathcal{P}$ has no corner cuts, then $(\mathfrak{X},\mathfrak{B})=(\mathfrak{X}^\bullet,\mathfrak{B}^\bullet)$.

\begin{remark}
\label{independentfromthechoiceofthefamily}
With the same notation introduced above, denote by $(X,B)$ (resp. $(X^\bullet,B^\bullet)$) the central fiber of $(\mathfrak{X},\mathfrak{B})$ (resp. $(\mathfrak{X}^\bullet,\mathfrak{B}^\bullet)$). Then $(X^\bullet,B^\bullet)$ only depends on $(X,B)$ and it is independent from the whole family $(\mathfrak{X},\mathfrak{B})\rightarrow\Spec(A)$.
\end{remark}

\begin{definition}
\label{modificationcentralfiberbullet}
Let $(X,B)$ be a stable toric pair of type $Q$. Define $(X^\bullet,B^\bullet)$ to be the central fiber of $(\mathfrak{X}^\bullet,\mathfrak{B}^\bullet)$, where $(\mathfrak{X},\mathfrak{B})\rightarrow\Spec(A)$ is a one-parameter family of stable toric pairs with central fiber $(X,B)$ and smooth generic fiber (such a family exists because $\overline{\mathbf{M}}_Q$ is irreducible, hence $(X,B)$ is smoothable). The pair $(X^\bullet,B^\bullet)$ is well defined by Remark~\ref{independentfromthechoiceofthefamily}. Observe that, if $\mathcal{P}$ has no corner cuts, then $(X,B)=(X^\bullet,B^\bullet)$.
\end{definition}


\section{Analysis of stability}
\label{wallcrossingstudy}

In the current section we continue the strategy outlined at the beginning of \S\ref{modificationsofthepairs}. More precisely, with the same notation of \S\ref{themodifiedfamilywithbullet}, we show that $(B^\bullet,\left(\frac{1+\epsilon}{2}\right)\Delta^\bullet|_{B_P^\bullet})$ is a stable pair.

\subsection{Preliminaries}

\begin{notation}
Consider a stable toric pair $(X,B)$ of type $Q$ and let $\mathcal{P}$ be the corresponding polyhedral subdivision of $(Q,Q\cap\mathbb{Z}^3)$. Let $\amalg_{P\in\mathcal{P}_3}X_P\rightarrow X$ be the normalization of $X$, where $X_P$ is the toric variety corresponding to the polytope $P$. Then we denote by $\Delta_P$ the toric boundary of $X_P$, by $D_P\subseteq X_P$ the conductor divisor, and by $B_P$ the restriction to $X_P$ of the preimage of $B$ under the normalization morphism. Define $X_P^\bullet,\Delta_P^\bullet,D_P^\bullet,B_P^\bullet$ analogously with $(X^\bullet,B^\bullet)$ instead of $(X,B)$.
\end{notation}

\begin{proposition}
\label{toricboundarycartier}
Let $(X,B)$ be a stable toric pair of type $Q$ and let $\mathcal{P}$ be the associated polyhedral subdivision of $(Q,Q\cap\mathbb{Z}^3)$. If $\mathcal{P}$ does not contain a corner cut, then $\Delta$ is Cartier. If $\mathcal{P}$ contains a corner cut, then $\Delta$ is not $\mathbb{Q}$-Cartier.
\end{proposition}

\begin{proof}
Define a piecewise linear function on the normal fan of each maximal dimensional polytope $P\in\mathcal{P}$ as follows. If $u_\rho$ is the $\mathbb{Z}$-generator of a ray $\rho$, then associate to $u_\rho$ the integer $-d_\rho$, where $d_\rho$ is the lattice distance between the facet of $2P$ normal to the ray and the lattice point $(1,1,1)$. Observe that this number is $0$ if the facet contains $(1,1,1)$, and $-1$ otherwise. If $\mathcal{P}$ has no corner cuts, then this gives a Cartier divisor on $X$ equal to $\Delta$.

Now, assume that $\mathcal{P}$ contains a corner cut $P$. Denote by $R$ that unique polytope in $\mathcal{P}$ sharing exactly a facet with $P$. Let $\ell_1,\ell_2,\ell_3$ be the three edges of $P$ which do not contain the apex. Observe that $\ell_i$, $i=1,2,3$, can be contained in two or three maximal dimensional polytopes in $\mathcal{P}$, $P$ and $R$ included.

If some $\ell_i$ is contained in three maximal dimensional polytopes, then take a point $x\in X$ lying on the torus invariant line corresponding to $\ell_i$. If $\Delta$ is $\mathbb{Q}$-Cartier, then $m\Delta$ is given by the vanishing of one equation in an open neighborhood of $x$ for some $m>0$. However, the vanishing locus of this equation on $X_R$ has codimension $2$, which cannot be.

Assume that each $\ell_i$ is only contained in $P$ and $R$. Denote by $\nu\colon X^\nu\rightarrow X$ the normalization. If $\Delta$ is $\mathbb{Q}$-Cartier, then $(\nu^*\Delta)|_{X_R}=\Delta_R-D_R$ is also $\mathbb{Q}$-Cartier. But this is a contradiction because there is no $\mathbb{Q}$-piecewise linear function on the normal fan of $R$ corresponding to $\Delta_R-D_R$ (to see this, consider the normal cone to a vertex of $R$ in common with $P$).
\end{proof}

\begin{theorem}
\label{stabilityofthemodifiedpair}
Let $Q$ be the unit cube and let $(X,B)$ be a stable toric pair of type $Q$. Consider $(X^\bullet,B^\bullet)$ as in Definition~\ref{modificationcentralfiberbullet}. Then $\left(B^\bullet,\left(\frac{1+\epsilon}{2}\right)\Delta^\bullet|_{B^\bullet}\right)$ is a stable pair (we have that $\Delta^\bullet$ is $\mathbb{Q}$-Cartier by Proposition~\ref{toricboundarycartier}).
\end{theorem}


\subsection{Proof of Theorem~\ref{stabilityofthemodifiedpair}}
Let $\mathcal{P}$ be the polyhedral subdivision of $(Q,Q\cap\mathbb{Z}^3)$ associated to $(X,B)$. We show that for all $P\in\mathcal{P}^\bullet$, the pair $\left(B_P^\bullet,D_P^\bullet|_{B_P^\bullet}+\left(\frac{1+\epsilon}{2}\right)(\Delta_P^\bullet-D_P^\bullet)|_{B_P^\bullet}\right)$ is stable. As shown in Figure~\ref{polytopeswithoutcornercuts}, there are four possibilities for $P$ up to symmetries of $Q$.

\begin{definition}
\label{definitiontypepolytopeinpbullet}
We say that $P\in\mathcal{P}^\bullet$ has \emph{type} $(a)$ (resp. $(b),(c),(d)$) if $P$ is equal to the polytope in Figure~\ref{polytopeswithoutcornercuts} $(a)$ (resp. $(b),(c),(d)$) up to a symmetry of $Q$.
\end{definition}

\begin{figure}[H]
\centering
\includegraphics[scale=0.40,valign=t]{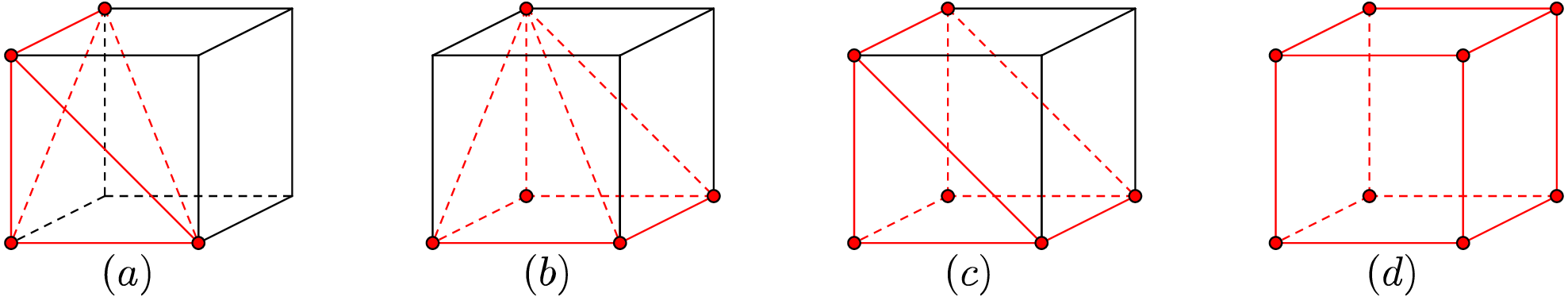}
\caption{Possible maximal dimensional polytopes in $\mathcal{P}^\bullet$ up to symmetries of $Q$.}
\label{polytopeswithoutcornercuts}
\end{figure}

\begin{proposition}
\label{adescription}
Given $P$ of type $(a)$, then $\left(B_P^\bullet,D_P^\bullet|_{B_P^\bullet}+\left(\frac{1+\epsilon}{2}\right)(\Delta_P^\bullet-D_P^\bullet)|_{B_P^\bullet}\right)$ is stable.
\end{proposition}

\begin{proof}
We have $X_P^\bullet=\mathbb{P}^3$ with coordinates $[W_0:\ldots:W_3]$, $D_P^\bullet=V(W_0)+V(W_1)$, $\Delta_P^\bullet-D_P^\bullet=V(W_2)+V(W_3)$, and $B_P^\bullet=V(a_0W_0+\ldots+a_3W_3)$, where $a_i\neq0$ for all $i=0,\ldots,3$. To find the equation of $B_P^\bullet$ we used Remark~\ref{combinatorialdescriptionstabletoricpair}, and $a_i\neq0$ for all $i$ because no corner cut can be contained in $P$ (recall the construction in \S\ref{themodifiedfamilywithbullet}). Then $B_P^\bullet$ is isomorphic to $\mathbb{P}^2$ and $\Delta_P^\bullet$ restricts to $B_P^\bullet$ giving four lines in general linear position, implying that $\left(B_P^\bullet,D_P^\bullet|_{B_P^\bullet}+\left(\frac{1+\epsilon}{2}\right)(\Delta_P^\bullet-D_P^\bullet)|_{B_P^\bullet}\right)$ is log canonical. Finally, if $L$ denotes a general line in $B_P^\bullet$, then
\begin{equation*}
K_{B_P^\bullet}+D_P^\bullet|_{B_P^\bullet}+\left(\frac{1+\epsilon}{2}\right)(\Delta_P^\bullet-D_P^\bullet)|_{B_P^\bullet}\sim\epsilon L,
\end{equation*}
which is ample.
\end{proof}

\begin{observation}
\label{threeincidentlinesweightslogcanonical}
The following standard fact will be useful in the analysis that follows. Let $L_1,L_2,L_3$ be three distinct concurrent lines in $\mathbb{A}^2$. Then the pair $\left(\mathbb{A}^2,\left(\frac{1+\epsilon}{2}\right)(L_1+L_2+L_3)\right)$ is log canonical.
\end{observation}

\begin{proposition}
\label{bdescription}
Given $P$ of type $(b)$, then $\left(B_P^\bullet,D_P^\bullet|_{B_P^\bullet}+\left(\frac{1+\epsilon}{2}\right)(\Delta_P^\bullet-D_P^\bullet)|_{B_P^\bullet}\right)$ is stable.
\end{proposition}

\begin{proof}
We have $X_P^\bullet=V(W_0W_1-W_2W_3)\subseteq\mathbb{P}^4$, $D_P^\bullet=V(W_0W_1-W_2W_3,W_0)$, $\Delta_P^\bullet-D_P^\bullet=V(W_0W_1-W_2W_3,W_1)+V(W_0W_1-W_2W_3,W_4)$. Finally, $B_P^\bullet=V(a_0W_0+\ldots+a_4W_4)\cap X_P^\bullet$, where $a_i\neq0$ for $i=1,\ldots,4$ and $a_0$ can possibly vanish (we can have at most one corner cut contained in $P$). $B_P^\bullet\cong\mathbb{P}^1\times\mathbb{P}^1$ because it is a hyperplane section of the projective cone $X_P^\bullet$ which does not pass through the vertex.

Now let us study the restrictions of $D_P^\bullet$ and $\Delta_P^\bullet-D_P^\bullet$ to $B_P^\bullet$. This boils down to understand how the coordinate hyperplanes $H_i=V(W_i)$, $i=0,\ldots,4$, restrict to $B_P^\bullet$. First of all, observe that $H_i$ cuts on $B_P^\bullet$ a curve $C$ of divisor class $(1,1)$. To show this, denote by $(a,b)$ the divisor class of $C=B_P^\bullet\cap H_i$. If $H'=V(a_0W_0+\ldots+a_4W_4)$ and $H$ is a general hyperplane in $\mathbb{P}^4$, then the self-intersection of $C$ is given by
\begin{equation*}
C^2=H_i|_{B_P^\bullet}\cdot H_i|_{B_P^\bullet}=H_i\cdot H_i\cdot B_P^\bullet=H_i\cdot H_i\cdot X_P^\bullet\cdot H'=H_i\cdot H_i\cdot2H\cdot H'=2.
\end{equation*}
On the other hand, $C^2=(a,b)^2=2ab=2$, implying that $(a,b)=(1,1)$. For $i\neq4$, $H_i\cap B_P^\bullet$ is always reducible, so it is given by two curves of divisor classes $(1,0)$ and $(0,1)$. If $a_0\neq0$, then it is easy to check that $H_4\cap B_P^\bullet$ is smooth for a general choice of the coefficients, but it can possibly break into two curves. In any case, $\Delta_P^\bullet$ restricts to $B_P^\bullet$ giving a simple normal crossing divisor, implying that $\left(B_P^\bullet,D_P^\bullet|_{B_P^\bullet}+\left(\frac{1+\epsilon}{2}\right)(\Delta_P^\bullet-D_P^\bullet)|_{B_P^\bullet}\right)$ is log canonical.

If $a_0=0$, then $H_4\cap B_P^\bullet$ is irreducible and it passes through the singular point of $H_1\cap B_P^\bullet$. In this case we have that $\left(B_P^\bullet,D_P^\bullet|_{B_P^\bullet}+\left(\frac{1+\epsilon}{2}\right)(\Delta_P^\bullet-D_P^\bullet)|_{B_P^\bullet}\right)$ is log canonical by Observation~\ref{threeincidentlinesweightslogcanonical}.

Finally, observe that
\begin{equation*}
K_{B_P^\bullet}+D_P^\bullet|_{B_P^\bullet}+\left(\frac{1+\epsilon}{2}\right)(\Delta_P^\bullet-D_P^\bullet)|_{B_P^\bullet}\sim(-2,-2)+(1,1)+\left(\frac{1+\epsilon}{2}\right)(2,2)=\epsilon(1,1),
\end{equation*}
which is ample.
\end{proof}

\begin{proposition}
\label{cdescription}
Given $P$ of type $(c)$, then $\left(B_P^\bullet,D_P^\bullet|_{B_P^\bullet}+\left(\frac{1+\epsilon}{2}\right)(\Delta_P^\bullet-D_P^\bullet)|_{B_P^\bullet}\right)$ is stable.
\end{proposition}

\begin{proof}
We have $X_P^\bullet=\mathbb{P}^2\times\mathbb{P}^1$ with coordinates $([X_0:X_1:X_2],[Y_0:Y_1])$, $D_P^\bullet=V(X_2)$, $\Delta_P^\bullet-D_P^\bullet=V(X_0X_1Y_0Y_1)$, and $B_P^\bullet=V((a_0X_0+a_1X_1+a_2X_2)Y_0+(b_0X_0+b_1X_1+b_2X_2)Y_1)$, where $a_0,a_1,b_0,b_1\neq0$ and at most one among $a_2$ and $b_2$ can be zero because there is at most one corner cut contained in $P$. Let us start by assuming that $a_2b_2\neq0$.

If $B_P^\bullet$ is singular, then one can show that $(b_0,b_1,b_2)=\lambda(a_0,a_1,a_2)$ for some $\lambda\in\mathbb{C}^*$. Therefore, the equation of $B_P^\bullet$ becomes
\begin{equation*}
(a_0X_0+a_1X_1+a_2X_2)(Y_0+\lambda Y_1)=0,
\end{equation*}
where $V(Y_0+\lambda Y_1)\cong\mathbb{P}^2$ and $V(a_0X_0+a_1X_1+a_2X_2)\cong\mathbb{P}^1\times\mathbb{P}^1$ are glued along a ruling of $\mathbb{P}^1\times\mathbb{P}^1$ and the line $a_0X_0+a_1X_1+a_2X_2=0$ in $\mathbb{P}^2$. The restrictions of $D_P^\bullet$ and $\Delta_P^\bullet-D_P^\bullet$ to these two irreducible components are described in Remark~\ref{explicitdescriptionirreduciblecomponentsdegenerations} $(c3)$. In this case, to conclude that $\left(B_P^\bullet,D_P^\bullet|_{B_P^\bullet}+\left(\frac{1+\epsilon}{2}\right)(\Delta_P^\bullet-D_P^\bullet)|_{B_P^\bullet}\right)$ is stable, we reduce the question to each connected component of the normalization of $B_P^\bullet$ and we apply what we already proved in the cases $(a)$ and $(b)$ above.

Now let us assume that $B_P^\bullet$ is smooth (and hence irreducible). By the discussion above, the two vectors $(a_0,a_1,a_2)$ and $(b_0,b_1,b_2)$ are not proportional. Denote by $p$ the point of intersection of the two lines $a_0X_0+a_1X_1+a_2X_2=0$ and $b_0X_0+b_1X_1+b_2X_2=0$ in $\mathbb{P}^2$. If $\pi\colon B_P^\bullet\rightarrow\mathbb{P}^2$ is the restriction to $B_P^\bullet$ of the projection map $\mathbb{P}^2\times\mathbb{P}^1\rightarrow\mathbb{P}^2$, then $\pi$ restricted to the complement of $\pi^{-1}(p)$ is an isomorphism, and $\pi^{-1}(p)\cong\mathbb{P}^1$. This proves that $B_P^\bullet\cong\mathbb{F}_1$. In this case, let us explain how $D_P^\bullet|_{B_P^\bullet}$ depends on the coefficients $a_i,b_j$. The restriction $D_P^\bullet|_{B_P^\bullet}$ has equation
\begin{displaymath}
\left\{ \begin{array}{ll}
X_2=0,\\
(a_0X_0+a_1X_1)Y_0+(b_0X_0+b_1X_1)Y_1=0.
\end{array} \right.
\end{displaymath}
By an argument analogous to what we did for $B_P^\bullet$, we have that this restriction is irreducible if and only if $(a_0,a_1)$ and $(b_0,b_1)$ are not proportional. In this case, $D_P^\bullet|_{B_P^\bullet}$ is a section of $B_P^\bullet$ with self-intersection $1$. If $(a_0,a_1)$ and $(b_0,b_1)$ are proportional, then the equation of $D_P^\bullet|_{B_P^\bullet}$ becomes
\begin{displaymath}
\left\{ \begin{array}{ll}
X_2=0,\\
(a_0X_0+a_1X_1)(Y_0+\lambda Y_1)=0
\end{array} \right.
\end{displaymath}
for some $\lambda\in\mathbb{C}^*$. The irreducible component $V(X_2,a_0X_0+a_1X_1)$ (resp. $V(X_2,Y_0+\lambda Y_1)$) is the exceptional section (resp. a fiber) of $B_P^\bullet$. We are left with understanding $(\Delta_P^\bullet-D_P^\bullet)|_{B_P^\bullet}$, and for this we need to study how $V(X_i),V(Y_i)$, $i=0,1$, restrict to $B_P^\bullet$. But $V(Y_i)$ restricts giving a fiber, and we can study $V(X_i)|_{B_P^\bullet}$ in the same way we did for $D_P^\bullet|_{B_P^\bullet}$. Observe that at most one among $V(X_0)|_{B_P^\bullet},V(X_1)|_{B_P^\bullet},D_P^\bullet|_{B_P^\bullet}$ can be reducible (otherwise $B_P^\bullet$ would be reducible). We conclude that $\left(B_P^\bullet,D_P^\bullet|_{B_P^\bullet}+\left(\frac{1+\epsilon}{2}\right)(\Delta_P^\bullet-D_P^\bullet)|_{B_P^\bullet}\right)$ is log canonical.

Now we consider the case where exactly one among $a_2$ and $b_2$ is zero. It is easy to check that $B_P^\bullet$ is automatically smooth and a description similar to the one above applies, with the only difference that the restriction $(\Delta_P^\bullet-D_P^\bullet)|_{B_P^\bullet}$ can acquire a triple intersection point. In this case we know that $\left(B_P^\bullet,D_P^\bullet|_{B_P^\bullet}+\left(\frac{1+\epsilon}{2}\right)(\Delta_P^\bullet-D_P^\bullet)|_{B_P^\bullet}\right)$ is log canonical by Observation~\ref{threeincidentlinesweightslogcanonical}.

For the ampleness condition, let $h$ be a section of self-intersection $1$ on $B_P^\bullet\cong\mathbb{F}_1$ and let $f$ be a fiber. Then
\begin{equation*}
K_{B_P^\bullet}+D_P^\bullet|_{B_P^\bullet}+\left(\frac{1+\epsilon}{2}\right)(\Delta_P^\bullet-D_P^\bullet)|_{B_P^\bullet}\sim-2h-f+h+\left(\frac{1+\epsilon}{2}\right)(2h+2f)=\epsilon(h+f),
\end{equation*}
which is ample.
\end{proof}

\begin{proposition}
\label{ddescription}
Given $P$ of type $(d)$, then $\left(B_P^\bullet,D_P^\bullet|_{B_P^\bullet}+\left(\frac{1+\epsilon}{2}\right)(\Delta_P^\bullet-D_P^\bullet)|_{B_P^\bullet}\right)$ is stable.
\end{proposition}

\begin{proof}
We have $X_P^\bullet=(\mathbb{P}^1)^3$ with coordinates $([X_0:X_1],[Y_0:Y_1],[Z_0:Z_1])$, $D_P^\bullet=\emptyset$, $\Delta_P^\bullet=V(X_0X_1Y_0Y_1Z_0Z_1)$, and $B_P^\bullet=V\left(\sum_{i,j,k=0,1}c_{ijk}X_iY_jZ_k\right)$, where any two distinct coefficients $c_{ijk}$ and $c_{i'j'k'}$ cannot be simultaneously zero if $(i,j,k)$ and $(i',j',k')$ are vertices of the same edge of the cube. This is because inside $Q$ we cannot fit two corner cuts with apices lying on the same edge.

Let us first assume that $B_P^\bullet$ is smooth (hence irreducible), which holds for a general choice of the coefficients $c_{ijk}$ by Bertini's Theorem. Then the anticanonical class $-K_{B_P^\bullet}=-(K_{(\mathbb{P}^1)^3}+B_P^\bullet)|_{B_P^\bullet}=(1,1,1)|_{B_P^\bullet}$ is ample and $K_{B_P^\bullet}^2=6$, implying that $B_P^\bullet\cong\Bl_3\mathbb{P}^2$ (see \cite[Exercise V.21(1)]{Bea96}). If all the $c_{ijk}$ are nonzero, then the restriction $\Delta_P^\bullet|_{B_P^\bullet}$ can be as in Proposition~\ref{mainconstruction} (general case), or some of these lines can break into the union of two incident $(-1)$-curves. If some coefficients $c_{ijk}$ are zero, then the lines configuration $\Delta_P^\bullet|_{B_P^\bullet}$ acquires triple intersection points. In any case, the pair $\left(B_P^\bullet,\left(\frac{1+\epsilon}{2}\right)\Delta_P^\bullet|_{B_P^\bullet}\right)$ is log canonical, where we use Observation~\ref{threeincidentlinesweightslogcanonical} in the case of triple intersection points.

Now assume that $B_P^\bullet$ is irreducible and singular. Let $p\in B_P^\bullet$ be a singular point. We prove that (1) $p$ is a singularity of type $A_1$, (2) that $p$ is the only singular point of $B_P^\bullet$, and (3) that $p$ lies on at most one irreducible component of $\Delta_P^\bullet$. We can assume that $p$ is in the form $([1:a],[1:b],[1:c])$. The invertible change of coordinates $X_0'=X_0$, $X_1'=X_1-aX_0$ and so on, sends $B_P^\bullet$ to an isomorphic surface $B_P'$ which is singular at $p'=([1:0],[1:0],[1:0])$. If we set $x'=\frac{X_1'}{X_0'},y'=\frac{Y_1'}{Y_0'},z'=\frac{Z_1'}{Z_0'}$, then the equation of $B_P'$ in this affine patch is in the form
\begin{equation*}
c_0x'y'+c_1x'z'+c_2y'z'+c_3x'y'z'=0.
\end{equation*}
The coefficients $c_0,c_1,c_2$ are nonzero because $B_P'$ is irreducible. Therefore it is clear that the singularity is of type $A_1$, proving (1). After homogeneizing the above equation and by checking all the affine patches, we can see that $p'$ is the only singularity of $B_P'$, implying (2). Finally, if we assume that $p$ lies on two irreducible components of $\Delta_P^\bullet$, then we can compute that two coefficients $c_{ijk},c_{i'j'k'}$ are zero with $(i,j,k),(i',j',k')$ adjacent vertices of the cube, which cannot be. So (3) holds as well. To prove that $\left(B_P^\bullet,\left(\frac{1+\epsilon}{2}\right)\Delta_P^\bullet|_{B_P^\bullet}\right)$ is log canonical, we show that $\left(X_P^\bullet,\left(\frac{1+\epsilon}{2}\right)\Delta_P^\bullet+B_P^\bullet\right)$ is log canonical and then we use inversion of adjunction (see \cite{Kaw07}). This is done in two steps.

\begin{itemize}

\item[(i)] First we show that $\left(X_P^\bullet,\left(\frac{1+\epsilon}{2}\right)\Delta_P^\bullet+B_P^\bullet\right)$ is log canonical in a neighborhood of a quadruple intersection point $q$ of $\Delta_P^\bullet+B_P^\bullet$. Note that in this case $q$ must be different from the singular point $p$ by (3) above. Assume without loss of generality that $q=([1:0],[1:0],[1:0])$, and therefore $c_{000}=0$. In an affine neighborhood of $q$ the equation of $B_P^\bullet$ becomes
\begin{equation*}
c_{100}x+c_{010}y+c_{001}z+c_{110}xy+c_{101}xz+c_{011}yz+c_{111}xyz=0,
\end{equation*}
where we must have $c_{100},c_{010},c_{001}$ nonzero. The affine equations for $\Delta_P^\bullet$ at $q$ are $x=0,y=0,z=0$. Therefore, locally at $q$, the four irreducible components of $\Delta_P^\bullet+B_P^\bullet$ are equivalent to hyperplanes in general linear position. It is a standard calculation to show that this singularity of $\left(X_P^\bullet,\left(\frac{1+\epsilon}{2}\right)\Delta_P^\bullet+B_P^\bullet\right)$ is log canonical.

\item[(ii)] $\left(X_P^\bullet,\left(\frac{1+\epsilon}{2}\right)\Delta_P^\bullet+B_P^\bullet\right)$ is log canonical in the complement of the quadruple intersection points. This is true in the complement of $\Delta_P^\bullet$ because $B_P^\bullet$ has at most an $A_1$ singularity. Let $H\subseteq\Delta_P^\bullet$ be an irreducible component. We show that $\left(X_P^\bullet,\Delta_P^\bullet+B_P^\bullet\right)$ is log canonical in a neighborhood of $H$ away from the quadruple intersection points. But this follows from inversion of adjunction because $(H,(\Delta_P^\bullet-H+B_P^\bullet)|_H)$ is log canonical away from the quadruple intersection points. More in detail, $H\cong\mathbb{P}^1\times\mathbb{P}^1$, $(\Delta_P^\bullet-H)|_H$ gives the toric boundary of $H$, and $B_P^\bullet|_H$ is a $(1,1)$-curve with no components in common with the toric boundary.

\end{itemize}

We are left with the case $B_P^\bullet$ reducible. Up to symmetries, decompositions into two irreducible components are given by
\begin{equation*}
(aX_0Y_0+bX_0Y_1+cX_1Y_0+dX_1Y_1)(eZ_0+fZ_1)=0,
\end{equation*}
where the coefficients are nonzero and satisfy $ad\neq bc$. Both irreducible components are isomorphic to $\mathbb{P}^1\times\mathbb{P}^1$. Decompositions into three irreducible components are given by
\begin{equation*}
(aX_0+bX_1)(cY_0+dY_1)(eZ_0+fZ_1)=0,
\end{equation*}
where the coefficients are nonzero. Note that up to $\Aut((\mathbb{P}^1)^3)$ there is only one choice of such coefficients. The three irreducible components are isomorphic to $\mathbb{P}^1\times\mathbb{P}^1$. In both cases, information about how the irreducible components are glued together and how $\Delta_P^\bullet$ restricts to these can be found in Remark~\ref{explicitdescriptionirreduciblecomponentsdegenerations}~$(d2)$ and $(d3)$. From these observations we can argue that $\left(B_P^\bullet,\left(\frac{1+\epsilon}{2}\right)\Delta_P^\bullet|_{B_P^\bullet}\right)$ is semi-log canonical.

To conclude, $K_{B_P^\bullet}+\left(\frac{1+\epsilon}{2}\right)\Delta_P^\bullet|_{B_P^\bullet}$ is ample because it is the pullback to $B_{P}^\bullet$ of
\begin{equation*}
K_{X_P^\bullet}+B_P^\bullet+\left(\frac{1+\epsilon}{2}\right)\Delta_P^\bullet\sim\epsilon(1,1,1),
\end{equation*}
which is ample.
\end{proof}

The last proposition concludes the proof of Theorem~\ref{stabilityofthemodifiedpair}.\qed

\begin{remark}
\label{explicitdescriptionirreduciblecomponentsdegenerations}
The proof of Theorem~\ref{stabilityofthemodifiedpair} gives an explicit description of the possible stable pairs $\left(B_P^\bullet,D_P^\bullet|_{B_P^\bullet}+\left(\frac{1+\epsilon}{2}\right)(\Delta_P^\bullet-D_P^\bullet)|_{B^\bullet}\right)$ for all the stable toric pairs $(X,B)$ of type $Q$. Here we summarize these possibilities. In Figure~\ref{explicitdegenerationsinvariantwallcrossing}, a triangle (resp. trapezoid, parallelogram) means $B_P^\bullet\cong\mathbb{P}^2$ (resp. $\mathbb{F}^1$, $\mathbb{P}^1\times\mathbb{P}^1$). $D_P^\bullet|_{B_P^\bullet}$ is represented by the thickened segments and $(\Delta_P^\bullet-D_P^\bullet)|_{B_P^\bullet}$ by the colored segments. First, let us assume that $\mathcal{P}$ has no corner cuts, so that $(X,B)=(X^\bullet,B^\bullet)$. As a consequence of this, in the cases that follow the divisor $\Delta_P|_{B_P}$ is simple normal crossing.
\begin{itemize}
\item[$(a)$] $B_P\cong\mathbb{P}^2$ and $D_P|_{B_P}$ (resp. $(\Delta_P-D_P)|_{B_P}$) consists of two lines;
\item[$(b)$] $B_P\cong\mathbb{P}^1\times\mathbb{P}^1$ and $D_P|_{B_P}$ consists of two incident rulings. $(\Delta_P-D_P)|_{B_P}$ is given by two incident rulings and a curve of divisor class $(1,1)$ which can possibly be reducible;
\item[$(c1)$] $B_P\cong\mathbb{F}_1$ and $D_P|_{B_P}$ is a line disjoint from the exceptional divisor. $(\Delta_P-D_P)|_{B_P}$ is given by two fibers and two lines disjoint from the exceptional divisor. Exactly one of these last two lines can possibly break into the union of the exceptional divisor and a fiber;
\item[$(c2)$] $B_P\cong\mathbb{F}_1$ and $D_P|_{B_P}$ is the union of the exceptional divisor and a fiber. $(\Delta_P-D_P)|_{B_P}$ is given by two fibers and two lines disjoint from the exceptional divisor;
\item[$(c3)$] $B_P$ is isomorphic to the union of $\mathbb{P}^2$ and $\mathbb{P}^1\times\mathbb{P}^1$ glued along a line in $\mathbb{P}^2$ and a ruling in $\mathbb{P}^1\times\mathbb{P}^1$. $D_P|_{B_P}$ consists of a line in $\mathbb{P}^2$ and a ruling in $\mathbb{P}^1\times\mathbb{P}^1$. $(\Delta_P-D_P)|_{B_P}$ is given by two lines on the $\mathbb{P}^2$ component and four rulings on $\mathbb{P}^1\times\mathbb{P}^1$ arranged as shown in Figure~\ref{explicitdegenerationsinvariantwallcrossing} $(c3)$;
\item[$(d1)$] $B_P\cong\Bl_3\mathbb{P}^2$ and $D_P|_{B_P}=\emptyset$. $\Delta_P|_{B_P}$ is as in Proposition~\ref{mainconstruction} (this is the general case), or some lines can possibly break into two intersecting $(-1)$-curves;
\item[$(d1')$] $B_P$ is a singular del Pezzo surface of degree $6$ with exactly one $A_1$ singularity. This singularity can lie on at most one irreducible component of $\Delta_P|_{B_P}$;
\item[$(d2)$] $B_P$ is isomorphic to the union of two copies of $\mathbb{P}^1\times\mathbb{P}^1$ glued along a ruling and an irreducible curve of divisor class $(1,1)$. $\Delta_P|_{B_P}$ is given by four rulings on one component and six rulings on the other. These are arranged as shown in Figure~\ref{explicitdegenerationsinvariantwallcrossing} $(d2)$;
\item[$(d3)$] $B_P$ is isomorphic to the union of three copies of $\mathbb{P}^1\times\mathbb{P}^1$ glued along rulings as shown in Figure~\ref{explicitdegenerationsinvariantwallcrossing} $(d3)$. $\Delta_P|_{B_P}$ consists of four rulings on each component as shown in the same figure.
\end{itemize}
Now assume that $\mathcal{P}$ has a corner cut and let $P\in\mathcal{P}^\bullet$. In this case, the possibilities for the pair $\left(B_P^\bullet,D_P^\bullet|_{B_P^\bullet}+\left(\frac{1+\epsilon}{2}\right)(\Delta_P^\bullet-D_P^\bullet)|_{B_P^\bullet}\right)$ are as above, with the difference that $(\Delta_P^\bullet-D_P^\bullet)|_{B_P^\bullet}$ is allowed to have triple intersection points.
\end{remark}

\begin{figure}
\centering
\includegraphics[scale=0.70,valign=t]{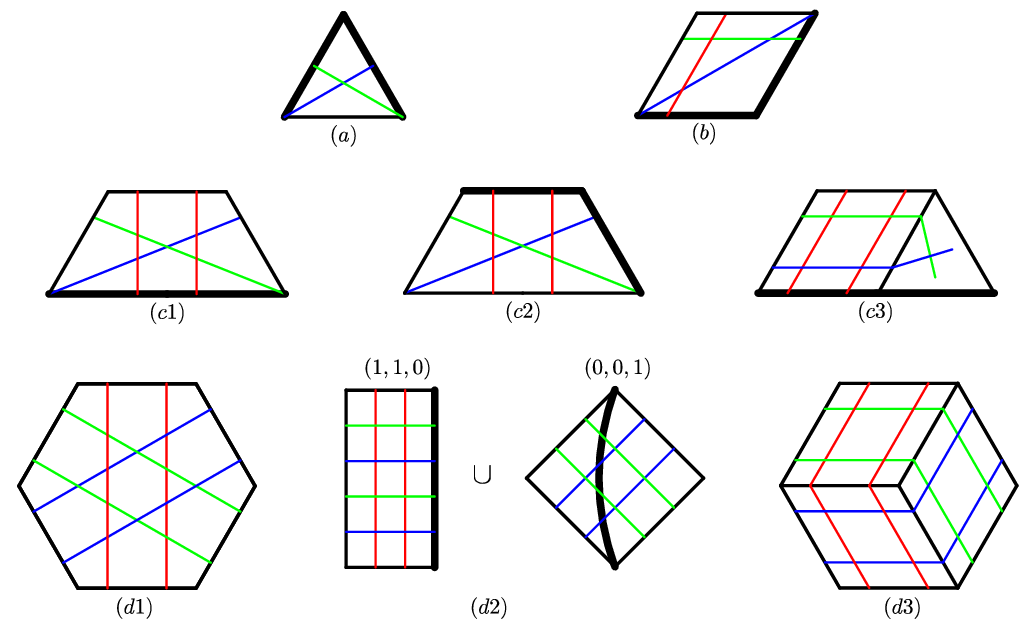}
\caption{The pictures represent $\left(B_P^\bullet,D_P^\bullet|_{B_P^\bullet}+\left(\frac{1+\epsilon}{2}\right)(\Delta_P^\bullet-D_P^\bullet)|_{B_P^\bullet}\right)$ for $P\in\mathcal{P}^\bullet$ (see Remark~\ref{explicitdescriptionirreduciblecomponentsdegenerations}). The restriction $(\Delta_P^\bullet-D_P^\bullet)|_{B_P^\bullet}$ is in color and $D_P^\bullet|_{B_P^\bullet}$ is thickened. In $(d2)$ the two surfaces are glued along the thickened curves.}
\label{explicitdegenerationsinvariantwallcrossing}
\end{figure}


\section{Study of the KSBA compactification}
\label{sectiononksbacompactification}

We are now ready to combine the results from \S\ref{modificationsofthepairs} and \S\ref{wallcrossingstudy} to prove that the rational map $\overline{\mathbf{M}}_Q\dashrightarrow\overline{\mathbf{M}}$ extends to the whole $\overline{\mathbf{M}}_Q$ (see Theorem~\ref{allyouneedtoknowaboutksbapoints}). As a consequence, in Theorem~\ref{propertiesofthemapstptoksba} we realize the stable pair compactification $\overline{\mathbf{M}}$ as a finite quotient of $\overline{\mathbf{M}}_Q$. Finally, in \S\ref{globalgeography} we study the structure of the boundary of $\overline{\mathbf{M}}$ and list the degenerations parametrized by it.


\subsection{The morphism $\overline{\mathbf{M}}_Q\rightarrow\overline{\mathbf{M}}$}

We start by recalling a result from \cite{GG14}.

\begin{notation}
\label{setupforlimits}
Consider a map $g\colon\Spec(K)\rightarrow Y$, where $K$ is the field of fractions of a DVR $(A,\mathfrak{m})$ with residue field our fixed base field $\mathbb{C}$. Let $Y$ be a proper scheme over a noetherian scheme $S$. By the valuative criterion of properness, $g$ uniquely extends to a map $\overline{g}\colon\Spec(A)\rightarrow Y$. We denote by $\lim g$ the point $\overline{g}(\mathfrak{m})$.
\end{notation}

\begin{theorem}[{\cite[Theorem 7.3]{GG14}}]
\label{lemmaextensionmapmodulispaces}
Suppose $X_1$ and $X_2$ are proper schemes over $\mathbb{C}$ with $X_1$ normal. Let $U\subseteq X_1$ be an open dense subset and consider a morphism $f\colon U\rightarrow X_2$. Then $f$ extends to a morphism $\overline{f}\colon X_1\rightarrow X_2$ if and only if, for any point $x\in X_1$, any DVR $(A,\mathfrak{m})$ as in Notation~\ref{setupforlimits}, and any morphism $g\colon\Spec(K)\rightarrow U$ such that $\lim g=x$, the point $\lim(f\circ g)$ in $X_2$ only depends on $x$, and not on the choice of $g$.
\end{theorem}

\begin{theorem}
\label{allyouneedtoknowaboutksbapoints}
Let $Q$ be the unit cube. Then there is a surjective morphism $\overline{\mathbf{M}}_Q\rightarrow\overline{\mathbf{M}}$ which on $\mathbb{C}$-points is given by
\begin{equation*}
(X,B)\mapsto\left(B^\bullet,\left(\frac{1+\epsilon}{2}\right)\Delta^\bullet|_{B^\bullet}\right).
\end{equation*}
\end{theorem}

\begin{proof}
Let $X_1=\overline{\mathbf{M}}_Q$ and $X_2=\overline{\mathbf{M}}$. Let $U=\mathbf{A}_{\sm}$, where $\mathbf{A}_{\sm}$ is the open subset of the torus $\mathbf{A}\subseteq\overline{\mathbf{M}}_Q$ in the proof of Proposition~\ref{rationalmapfromstptoksba} parametrizing stable toric pairs $((\mathbb{P}^1)^3,B)$ with $B$ smooth. We obtain a morphism $X_1\rightarrow X_2$ by extending $f\colon U\rightarrow X_2$ to the whole $X_1$ using Theorem~\ref{lemmaextensionmapmodulispaces}.

Let $x\in X$ and let $(A,\mathfrak{m})$ be any DVR. Consider a map $g\colon\Spec(K)\rightarrow U$ such that $\lim g=x$. Let $(X,B)$ be the stable toric pair of type $Q$ parametrized by $x$. Denote by $\Delta$ the toric boundary of $X$. If we prove that $\lim(f\circ g)$ corresponds to the stable pair $\left(B^\bullet,\left(\frac{1+\epsilon}{2}\right)\Delta^\bullet|_{B^\bullet}\right)$, then we are done by Theorem~\ref{lemmaextensionmapmodulispaces}, because this shows that $\lim(f\circ g)$ only depends on $x$, and not from the choice of $g$ (see Remark~\ref{independentfromthechoiceofthefamily}).

Let $(U\times(\mathbb{P}^1)^3,\mathcal{S})\rightarrow U$ be the restriction to $U$ of the universal family of stable toric pairs over $\mathbf{A}$ constructed in the proof of Proposition~\ref{rationalmapfromstptoksba}. Let $((\mathbb{P}_K^1)^3,\mathfrak{B}^\circ)$ be the pullback of $(U\times (\mathbb{P}^1)^3,\mathcal{S})$ under the map $g$ and denote by $(\mathfrak{X},\mathfrak{B})$ its completion over $\Spec(A)$, or a finite ramified base change of it, as a family of stable toric pairs. Note that by construction the central fiber of $(\mathfrak{X},\mathfrak{B})$ is $(X,B)$. Consider $(\mathfrak{X}^\bullet,\mathfrak{B}^\bullet)$ and, if $D$ denotes the toric boundary of $(\mathbb{P}^1)^3$, let $\overline{D}_K$ be the closure of $D_K=D\times\Spec(K)$ in $\mathfrak{X}^\bullet$. Then the central fiber of $\left(\mathfrak{B}^\bullet,\left(\frac{1+\epsilon}{2}\right)\overline{D}_K|_{\mathfrak{B}^\bullet}\right)$, which is $\left(B^\bullet,\left(\frac{1+\epsilon}{2}\right)\Delta^\bullet|_{B^\bullet}\right)$ by construction, is the stable pair corresponding to $\lim(f\circ g)$.
\end{proof}

\begin{theorem}
\label{propertiesofthemapstptoksba}
Let $\overline{\mathbf{M}}_Q\rightarrow\overline{\mathbf{M}}$ be the morphism in Theorem~\ref{allyouneedtoknowaboutksbapoints}. Then the induced morphism $\overline{\mathbf{M}}_Q/\Sym(Q)\rightarrow\overline{\mathbf{M}}$ is an isomorphism.
\end{theorem}

\begin{proof}
The group $\Sym(Q)$ acts on $\overline{\mathbf{M}}_Q$ because $\overline{\mathbf{M}}_Q$ is the projective toric variety associated to the secondary polytope of $(Q,Q\cap\mathbb{Z}^3)$. The modular interpretation of the action is the following: $\Sym(Q)$ acts on the stable toric pair $(X,B)$ by changing the torus action on it. In particular, the morphism $\overline{\mathbf{M}}_Q\rightarrow\overline{\mathbf{M}}$ is $\Sym(Q)$-equivariant. Therefore, we have an induced morphism $\overline{\mathbf{M}}_Q/\Sym(Q)\rightarrow\overline{\mathbf{M}}$.

The fibers of the restriction of $\overline{\mathbf{M}}_Q\rightarrow\overline{\mathbf{M}}$ to $\mathbf{A}_{\sm}$ are exactly $\Sym(Q)$-orbits by Proposition~\ref{mainconstruction}, and hence the morphism $\overline{\mathbf{M}}_Q/\Sym(Q)\rightarrow\overline{\mathbf{M}}$ is bijective on a dense open subset. In what follows, we show that $\overline{\mathbf{M}}_Q\rightarrow\overline{\mathbf{M}}$ is quasi-finite, which implies that $\overline{\mathbf{M}}_Q/\Sym(Q)\rightarrow\overline{\mathbf{M}}$ is an isomorphism by Zariski's Main Theorem.

To prove that $\overline{\mathbf{M}}_Q\rightarrow\overline{\mathbf{M}}$ is quasi-finite, it is enough to check that no $1$-dimensional boundary stratum of $\overline{\mathbf{M}}_Q$ is contracted. These strata correspond to the minimal elements of the poset of regular polyhedral subdivisions of $(Q,Q\cap\mathbb{Z}^3)$ which are not triangulations (recall that the polyhedral subdivisions of $(Q,Q\cap\mathbb{Z})$ are regular, see the proof of Proposition~\ref{irreducibilitycoarsemoduliofstabletoricpairs}). If $\mathcal{P}$ is one of these polyhedral subdivisions, then it contains a subpolytope of $Q$ with vertices in $Q\cap\mathbb{Z}^3$ which can be subdivided further only once. A simple enumeration shows that $\mathcal{P}$ has to contain one of the polytopes listed in Figure~\ref{threepolytopesinonedimensionalsubdivision}. Denote by $P$ one of such polytopes and let $(X,B)$ be a stable toric pair with $\mathcal{P}$ as corresponding polyhedral subdivision of $(Q,Q\cap\mathbb{Z}^3)$. Then, as $(X,B)$ varies among the stable toric pairs parametrized by the $1$-dimensional boundary stratum corresponding to $\mathcal{P}$, $(B^\bullet,\left(\frac{1+\epsilon}{2}\right)\Delta^\bullet|_{B^\bullet})$ describes a $1$-dimensional family of stable pairs. To see this, if $P^\bullet\in\mathcal{P}^\bullet$ is the polytope corresponding to $P\in\mathcal{P}$, then the irreducible component of $(B^\bullet,\left(\frac{1+\epsilon}{2}\right)\Delta^\bullet|_{B^\bullet})$ corresponding to $P^\bullet$ shows the variation of this parameter. In conclusion, the $1$-dimensional boundary stratum corresponding to $\mathcal{P}$ is not contracted.
\end{proof}

\begin{figure}
\centering
\includegraphics[scale=0.40,valign=t]{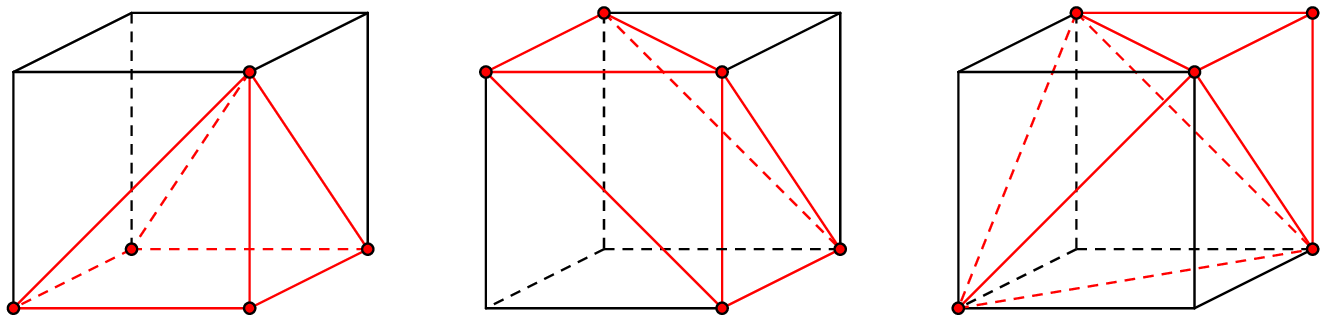}
\caption{Subpolytopes of $Q$ which can be subdivided further only once in the proof of Theorem~\ref{propertiesofthemapstptoksba}.}
\label{threepolytopesinonedimensionalsubdivision}
\end{figure}


\subsection{Description of the boundary of $\overline{\mathbf{M}}$}
\label{globalgeography}

\begin{definition}
The \emph{boundary} of the moduli space $\overline{\mathbf{M}}$ is the closed subset whose $\mathbb{C}$-points parametrize stable pairs $(B,D)$ where $B$ is reducible. Let $(B,D)$ be a stable pair parametrized by the boundary, and consider the locus of points in $\overline{\mathbf{M}}$ parametrizing stable pairs $(B',D')$ such that $B\cong B'$. We call the closure of such locus a \emph{stratum}.
\end{definition}

\begin{theorem}
\label{combinatoricsstratificationstatement}
The boundary of $\overline{\mathbf{M}}$ is stratified as shown in Figure~\ref{explicitstratification}. For each stratum we show the degeneration of $(\Bl_3\mathbb{P}^2,\frac{1+\epsilon}{2}\sum_{i=1}^3(\ell_i+\ell_i'))$ parametrized by a general point in the stratum. The strata are organized from bottom to top in increasing order of dimension. In particular, the boundary consists of three irreducible components: two divisors $\mathbf{D}_1,\mathbf{D}_2$ and a curve $\mathbf{C}$ with $\mathbf{D}_1\cap\mathbf{C}=\emptyset$. If $Q$ is the unit cube, the strata containing the leftmost $0$-dimensional stratum correspond bijectively to the polyhedral subdivisions of $(Q,Q\cap\mathbb{Z}^3)$ without corner cuts up to symmetries of $Q$.
\end{theorem}

\begin{proof}
By Theorem~\ref{propertiesofthemapstptoksba} any stable pair parametrized by $\overline{\mathbf{M}}$ is in the form $\left(B^\bullet,\left(\frac{1+\epsilon}{2}\right)\Delta^\bullet|_{B^\bullet}\right)$ for some stable toric pair $(X,B)$ parametrized by $\overline{\mathbf{M}}_Q$. In Remark~\ref{explicitdescriptionirreduciblecomponentsdegenerations} we listed all the possibilities for the pairs $\left(B_P^\bullet,D_P^\bullet|_{B_P^\bullet}+\left(\frac{1+\epsilon}{2}\right)(\Delta_P^\bullet-D_P^\bullet)|_{B_P^\bullet}\right)$, which can be glued together for $P\in\mathcal{P}^\bullet$ to recover $\left(B^\bullet,\left(\frac{1+\epsilon}{2}\right)\Delta^\bullet|_{B^\bullet}\right)$, where $\mathcal{P}^\bullet$ is a polyhedral subdivisions of $(Q,Q\cap\mathbb{Z}^3)$ without corner cuts. These subdivisions are shown in Figure~\ref{subdivisionscubewithoutcornercuts} up to symmetries of $Q$. The end result is shown in Figure~\ref{explicitstratification}, where for each stratum we show the stable pair parametrized by the general point in the stratum. The three colors used to draw the divisor have the following meaning: lines sharing the same color come from the same pair of lines on the original $\Bl_3\mathbb{P}^2$. Note that, even though we use three different colors, we do not distinguish the three pairs because we quotiented by $S_3$ in our definition of $\overline{\mathbf{M}}$. Thickened segments indicate lines along which two irreducible components are glued together. The claimed combinatorial interpretation of the strata containing the leftmost $0$-dimensional stratum follows after comparing Figure~\ref{subdivisionscubewithoutcornercuts} and Figure~\ref{explicitstratification}.
\end{proof}

\begin{figure}
\centering
\includegraphics[scale=0.25,valign=t]{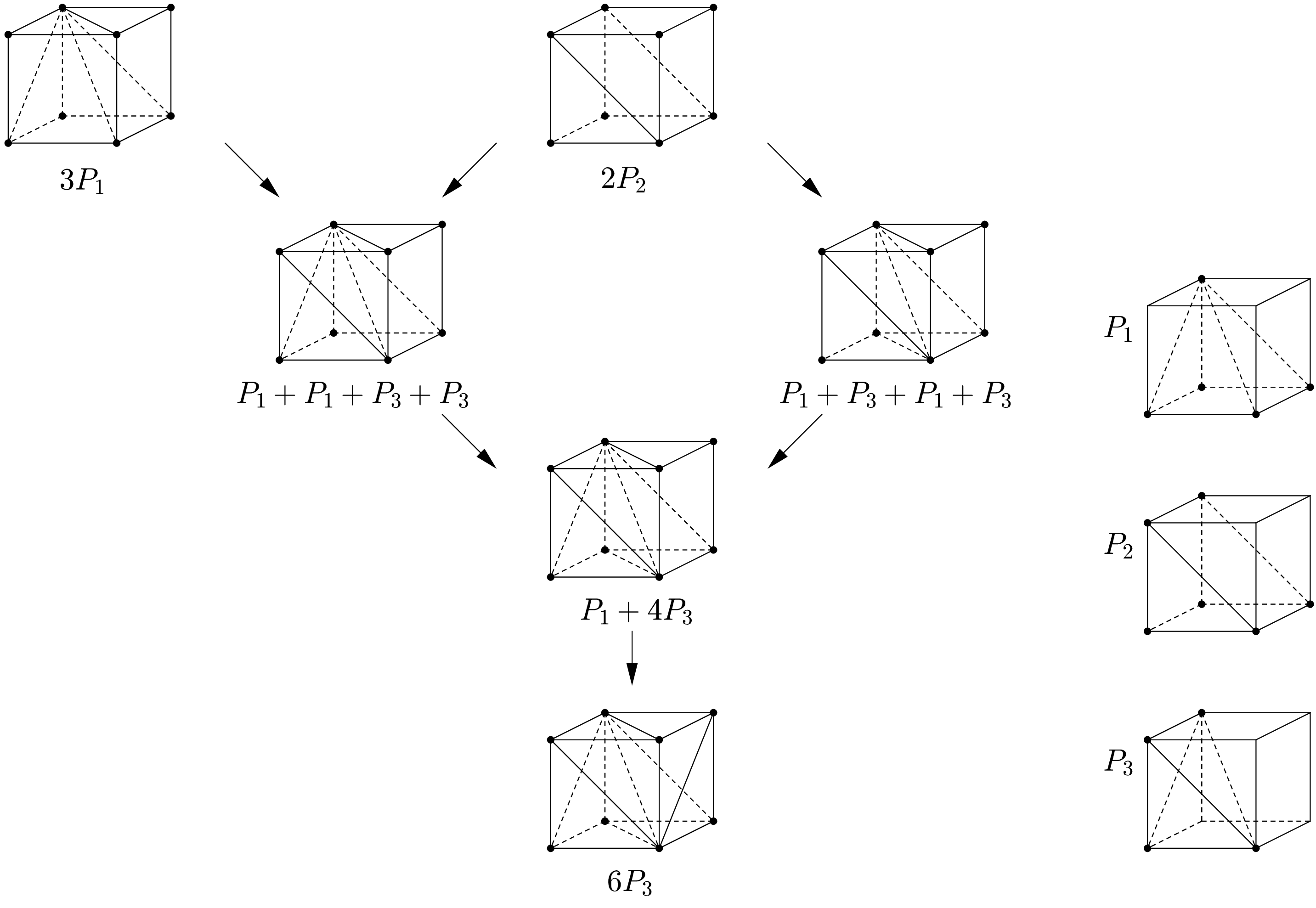}
\caption{Polyhedral subdivisions of $(Q,Q\cap\mathbb{Z}^3)$ without corner cuts up to symmetry and ordered by refinement. The arrows indicate refinement.}
\label{subdivisionscubewithoutcornercuts}
\end{figure}

\begin{figure}
\centering
\includegraphics[scale=0.40,valign=t]{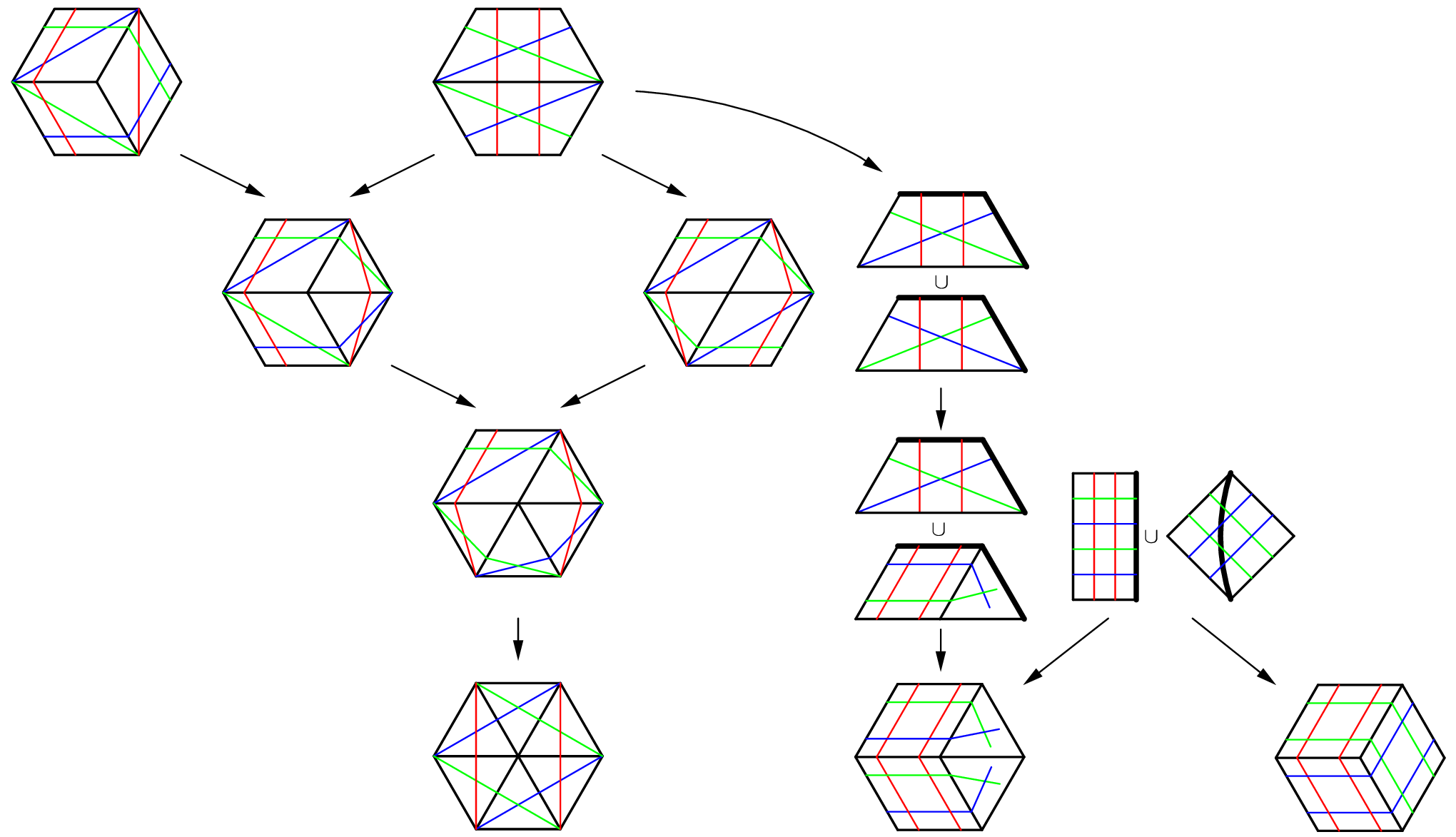}
\caption{Stratification of the boundary of $\overline{\mathbf{M}}$ and degenerations of $\Bl_3\mathbb{P}^2$ with the three pairs of weighted lines parametrized by it. The symbol ``$\cup$'' between two surfaces indicates that the two surfaces are glued together along the corresponding thickened curves. The arrows indicate specialization.}
\label{explicitstratification}
\end{figure}

\begin{remark}
\label{comparingdegenofsixlines}
We compare the generic degenerations of $\Bl_3\mathbb{P}^2$ together with the three pairs of weighted lines $\frac{1+\epsilon}{2}\sum_{i=1}^3(\ell_i+\ell_i')$ parametrized by the boundary of $\overline{\mathbf{M}}$ with the corresponding stable degenerations of six lines $\frac{1+\epsilon}{2}\sum_{i=1}^3(\overline{\ell}_i+\overline{\ell}_i')$ in $\mathbb{P}^2$ after contracting the three disjoint exceptional curves. Let $\mathbf{D}_1,\mathbf{D}_2$, and $\mathbf{C}$ be the irreducible boundary components of $\overline{\mathbf{M}}$ in Theorem~\ref{combinatoricsstratificationstatement}. Recall that the generic degenerations parametrized by these are in Figure~\ref{explicitstratification}.
\begin{enumerate}

\item The degeneration parametrized generically by $\mathbf{D}_1$ is obtained by letting the lines $\ell_1,\ell_2,\ell_3$ break into two $(-1)$-curves, so that in the limit there are two double $(-1)$-curves. This is illustrated in the top-left of Figure~\ref{comparisonwithsixlinesinP2}. The corresponding degeneration of $\mathbb{P}^2$ has $\overline{\ell}_1=\overline{\ell}_2$ and $\ell_3$ passes through $\overline{\ell}_1\cap\overline{\ell}_1'$. Its stable replacement is \cite[Figure~5.12~(11)]{Ale15}.

\item The degeneration parametrized generically by $\mathbf{D}_2$ is obtained by letting the lines $\ell_1,\ell_2$ break into the gluing of two $(-1)$-curves, so that in the limit there is a double $(-1)$-curve. This is illustrated in the top-middle of Figure~\ref{comparisonwithsixlinesinP2}. The corresponding degeneration of $\mathbb{P}^2$ has $\overline{\ell}_1=\overline{\ell}_2$, and its stable replacement is \cite[Figure~5.12~(5)]{Ale15}.

\item The degeneration parametrized generically by $\mathbf{C}$ is obtained by letting two lines in one pair collide, say $\ell_1=\ell_1'$ (see the top-right of Figure~\ref{comparisonwithsixlinesinP2}). Then the corresponding degeneration of $\mathbb{P}^2$ has $\overline{\ell}_1=\overline{\ell}_1'$, and the stable replacement is again \cite[Figure~5.12~(5)]{Ale15}.

\end{enumerate}
\end{remark}

\begin{figure}
\centering
\includegraphics[scale=0.40,valign=t]{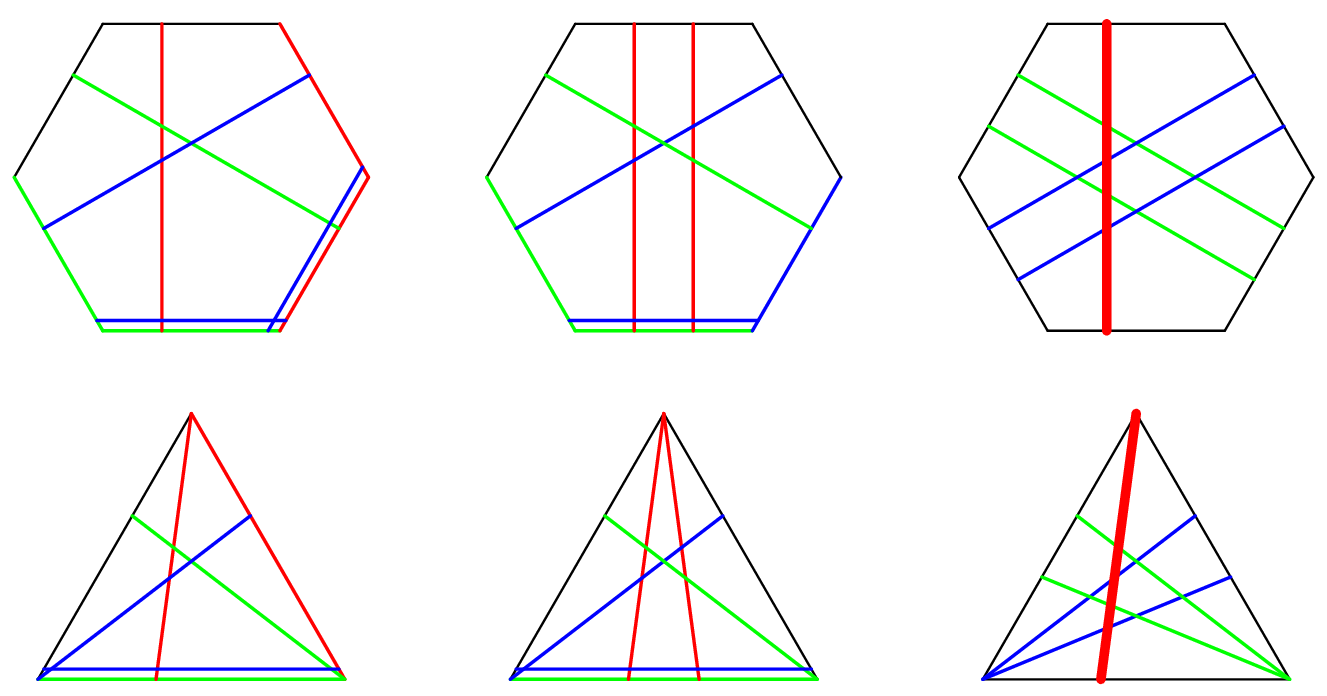}
\caption{Degenerations of line arrangements in $\Bl_3\mathbb{P}^2$ and $\mathbb{P}^2$ discussed in Remark~\ref{comparingdegenofsixlines}.}
\label{comparisonwithsixlinesinP2}
\end{figure}


\subsection{Degenerations of Enriques surfaces}
\label{studyofthecorrespondingz22cover}
After describing the degenerations of $\Bl_3\mathbb{P}^2$ together with the three pairs of weighted lines parametrized by $\overline{\mathbf{M}}$, we discuss the overlying $\mathbb{Z}_2^2$-covers, which correspond to degenerations of Enriques surfaces.

\begin{enumerate}

\item Consider the stable pair parametrized by the rightmost $0$-dimensional stratum in Figure~\ref{explicitstratification}. Then its $\mathbb{Z}_2^2$-cover is isomorphic to the quotient of
\begin{equation*}
V((X_0^2-X_1^2)(Y_0^2-Y_1^2)(Z_0^2-Z_1^2))\subseteq(\mathbb{P}^1)^3,
\end{equation*}
by the involution $\iota$ (see Definition~\ref{familyD16polarizedEnriquessurfaces} and the proof of Proposition~\ref{ddescription}). This quotient consists of three copies of $\mathbb{P}^1\times\mathbb{P}^1$ glued along rulings in such a way that its dual complex gives a triangulation of the real projective plane.

\item The stable pair parametrized by a general point in $\mathbf{C}$ has $\mathbb{Z}_2^2$-cover isomorphic to the quotient of
\begin{equation*}
V((X_0^2-X_1^2)(Y_0^2Z_0^2+Y_1^2Z_0^2+Y_0^2Z_1^2+\lambda Y_1^2Z_1^2))\subseteq(\mathbb{P}^1)^3,\lambda\neq0,1,
\end{equation*}
by $\iota$ (see Definition~\ref{familyD16polarizedEnriquessurfaces} and the proof of Proposition~\ref{ddescription}). If we define
\begin{equation*}
E=V(Y_0^2Z_0^2+Y_1^2Z_0^2+Y_0^2Z_1^2+\lambda Y_1^2Z_1^2)\subseteq\mathbb{P}^1\times\mathbb{P}^1,
\end{equation*}
then the irreducible components of this $\mathbb{Z}_2^2$-cover are a copy of $\mathbb{P}^1\times\mathbb{P}^1$, and an elliptic fibration $F$ over $\mathbb{P}^1$ with fibers isomorphic to $E$ and two double fibers. These surfaces are glued along $E\subseteq\mathbb{P}^1\times\mathbb{P}^1$ and a reduced fiber of $F$.

\item Consider the stable pair parametrized by a general point in the boundary divisor $\mathbf{D}_2$. Let us describe the $\mathbb{Z}_2^2$-cover $X$ of one of the two irreducible components, which are both isomorphic to $\mathbb{F}_1$. Let $h$ be a section of self-intersection $1$ and $f$ a fiber. Then the building data for the cover $\pi\colon X\rightarrow\mathbb{F}_1$ is given by $D_a\sim2f+h,D_b\sim D_c\sim h$, implying that
\begin{equation*}
K_X\sim_\mathbb{Q}\pi^*\left(K_{\mathbb{F}_1}+\frac{1}{2}(D_a+D_b+D_c)\right)\sim_\mathbb{Q}-\frac{1}{2}\pi^*(h).
\end{equation*}
This shows that $-K_X$ is big, nef, and $K_X^2=1$. Therefore, $X$ is a weak del Pezzo surface of degree $1$. The total degeneration is given by two of such weak del Pezzo surfaces glued along an elliptic curve.

\item Let $X$ be the $\mathbb{Z}_2^2$-cover of the $\mathbb{P}^1\times\mathbb{P}^1$ in Figure~\ref{explicitdegenerationsinvariantwallcrossing} $(b)$. Denote by $\ell_1$ and $\ell_2$ two incident rulings. Then the building data for the cover $\pi\colon X\rightarrow\mathbb{P}^1\times\mathbb{P}^1$ is given by $D_a\sim D_b\sim D_c\sim\ell_1+\ell_2$. From this we obtain that $K_X\sim_\mathbb{Q}-\frac{1}{2}\pi^*(\ell_1+\ell_2)$. It follows that $-K_X$ is ample and $K_X^2=2$. Hence, $X$ is a del Pezzo surface of degree $2$. Gluing together three of these gives the surface parametrized by a general point in $\mathbf{D}_1$.

\item Let us describe the $\mathbb{Z}_2^2$-cover of the $\mathbb{P}^2$ in Figure~\ref{explicitdegenerationsinvariantwallcrossing} $(a)$. If $\ell$ denotes a line in $\mathbb{P}^2$, then the building data for the cover $\pi\colon X\rightarrow\mathbb{P}^2$ is $D_a\sim D_b\sim2\ell,D_c\sim0$. Therefore, $K_X\sim_\mathbb{Q}-\pi^*(\ell)$, implying that $-K_X$ is ample and $K_X^2=4$. Hence, $X$ is a del Pezzo surface of degree $4$.

\item Let us describe the $\mathbb{Z}_2^2$-cover of the $\mathbb{P}^1\times\mathbb{P}^1$-component in Figure~\ref{explicitdegenerationsinvariantwallcrossing} $(c3)$. If $\ell_1$ and $\ell_2$ denote two incident rulings, then the building data for the cover $\pi\colon X\rightarrow\mathbb{P}^1\times\mathbb{P}^1$ is given by $D_a\sim2\ell_1+\ell_2,D_b\sim D_c\sim\ell_2$. Therefore, $K_X\sim_\mathbb{Q}-\frac{1}{2}\pi^*(2\ell_1+\ell_2)$, $-K_X$ is ample, $K_X^2=4$, and $X$ is a del Pezzo surface of degree $4$.

\end{enumerate}

We summarize part (2), (3), and (4) of the calculations above in the next corollary.

\begin{corollary}
\label{generaldegenerationsofenriques}
Let $\mathbf{D}_1,\mathbf{D}_2$, and $\mathbf{C}$ be the irreducible components of the boundary of $\overline{\mathbf{M}}$ as in Theorem~\ref{combinatoricsstratificationstatement}. Then the surfaces parametrized by the general point of $\mathbf{D}_1,\mathbf{D}_2$, and $\mathbf{C}$ are respectively
\begin{enumerate}
\item the gluing of three del Pezzo surfaces of degree $2$ so that the dual complex is a $2$-simplex and the double locus consists of three smooth rational curves;
\item the gluing of two weak del Pezzo surfaces of degree $1$ along an elliptic curve;
\item the gluing of $\mathbb{P}^1\times\mathbb{P}^1$ and an elliptic ruled surface along a $(2,2)$ curve and a reduced fiber.
\end{enumerate}
\end{corollary}

\begin{remark}
A Coble surface is a smooth rational projective surface $X$ with $|-K_X|=\emptyset$ and $|-2K_X|\neq\emptyset$ (see \cite{DZ01}). These are related to our degenerations of Enriques surfaces as follows. Let $S$ be the appropriate $\mathbb{Z}_2^2$-cover of $\Bl_3\mathbb{P}^2$ branched along $\sum_{i=1}^3(\ell_i+\ell_i')$, and assume that this line configuration has exactly one triple intersection point. Then $S$ has a quotient singularity of type $\frac{(1,1)}{4}$ over this triple intersection point, and the minimal resolution $\widetilde{S}$ of $S$ is a Coble surface. This follows from Castelnuovo rationality criterion, and from $K_{\widetilde{S}}=-\frac{1}{2}E$, where $E$ is the exceptional divisor over the quotient singularity. The Zariski closure in $\overline{\mathbf{M}}$ of the locus of points parametrizing these surfaces $S$ defines a divisor.
\end{remark}


\section{Morphism from KSBA to Baily--Borel compactification}
\label{relationwithbailyborel}
In what follows we construct a morphism from the KSBA compactification $\overline{\mathbf{M}}$ to the Baily--Borel compactification of $\mathcal{D}/\Gamma$, where $\mathcal{D}$ is the period domain parametrizing $D_{1,6}$-polarized Enriques surfaces (details in \S\ref{relationksbabailyborel}). \S\ref{generalizedtypestudy} contains a technical result which is fundamental to construct such morphism.


\subsection{Generalized type of degenerations of stable K3 surface pairs}
\label{generalizedtypestudy}

\begin{remark}
In \S\ref{relationwithbailyborel} our focus moves from Enriques surfaces to K3 surfaces. The reason is that in Theorem~\ref{morphismtobailyborelcompactification} we compute the limits of one-parameter families of $D_{1,6}$-polarized Enriques surfaces in the Baily--Borel compactification of $\mathcal{D}/\Gamma$. This is done by considering the corresponding K3 covers.
\end{remark}

Let $\Delta$ be the unit disk $\{t\in\mathbb{C}\mid|t|<1\}$ and let $\Delta^*=\Delta\setminus\{0\}$. We are interested in proper flat families $\mathfrak{X}^*\rightarrow\Delta^*$ with $\mathfrak{X}^*$ smooth and such that the fiber $\mathfrak{X}_t^*$ is a smooth K3 surface for all $t\in\Delta^*$. Equip $\mathfrak{X}^*$ with an effective relative Cartier divisor $\mathcal{H}^*$ such that $(\mathfrak{X}^*,\mathcal{H}^*)\rightarrow\Delta^*$ is a family of stable pairs. Let $\mathfrak{X}$ be a semistable degeneration with $K_\mathfrak{X}\sim0$ completing $\mathfrak{X}^*$ over $\Delta$ (see \cite{Kul77,PP81}). We call $\mathfrak{X}\rightarrow\Delta$ Kulikov degeneration for short. Recall that the central fiber $\mathfrak{X}_0$ can be of type I, II or III (see \cite[Theorem II]{Kul77}). In type II, denote by $j(\mathfrak{X}_0)$ the $j$-invariant of one of the mutually isomorphic elliptic double curves in $\mathfrak{X}_0$. Then define $\mathcal{H}$ to be the closure of $\mathcal{H}^*$ inside $\mathfrak{X}$ (note that $\mathcal{H}$ is flat over $\Delta$). On the other hand, we can define a second completion of $(\mathfrak{X}^*,\mathcal{H}^*)$ over $\Delta$, which we denote by $(\mathfrak{X}',\mathcal{H}')$, such that $(\mathfrak{X}_0',\epsilon\mathcal{H}_0')$ is a stable pair for $0<\epsilon\ll1$.

\begin{definition}
With the notation introduced above, define the \emph{dual graph} of $\mathfrak{X}_0'$ as follows. Draw a vertex $v_i$ for each irreducible component $V_i$ of $\mathfrak{X}_0'$. Then, given any two distinct irreducible components $V_i$ and $V_j$, draw one edge between $v_i$ and $v_j$ for each irreducible curve in $V_i\cap V_j$. If an irreducible component $V_i$ self-intersects along a curve $C$, then draw one loop on $v_i$ for each irreducible component of $C$. Denote by $G(\mathfrak{X}_0')$ the dual graph of $\mathfrak{X}_0'$.
\end{definition}

\begin{definition}
\label{ksbatype}
Let $\mathfrak{X}'$ as defined above. We say that $\mathfrak{X}_0'$ has \emph{generalized type I, II}, or \emph{III} if the following hold:
\begin{itemize}
\item Type I: $G(\mathfrak{X}_0')$ consists of one vertex and $\mathfrak{X}_0'$ has at worst Du Val singularities;
\item Type II: $G(\mathfrak{X}_0')$ is a chain and $\mathfrak{X}_0'$ has at worst elliptic singularities. If there are at least two vertices and the double curves are mutually isomorphic elliptic curves, then denote by $j(\mathfrak{X}_0')$ the $j$-invariant of one of these;
\item Type III: otherwise.
\end{itemize}
\end{definition}

The proof of the following theorem, which builds upon the proof of \cite[Theorem 2.9]{Laz16}, was communicated to me by Valery Alexeev.

\begin{theorem}
\label{goodbehaviortypeofdegeneration}
With the notation introduced above, $\mathfrak{X}_0$ and $\mathfrak{X}_0'$ have the same type. In addition, if $\mathfrak{X}_0,\mathfrak{X}_0'$ have type II and $j(\mathfrak{X}_0')$ can be defined, then $j(\mathfrak{X}_0)=j(\mathfrak{X}_0')$.
\end{theorem}

\begin{proof}
The proof of \cite[Theorem 2.9]{Laz16} describes a procedure to construct the unique stable model $(\mathfrak{X}',\epsilon\mathcal{H}')$ by modifying $(\mathfrak{X},\mathcal{H})$. This procedure consists of the following two steps:
\begin{itemize}
\item Step $1$: Replace $\mathfrak{X}$ with another Kulikov degeneration such that $\mathcal{H}$ is nef and does not contain double curves or triple points. This may involve, among other things, base changes;
\item Step $2$: For $n\geq4$, the line bundle $\mathcal{O}_\mathfrak{X}(n\mathcal{H})$ induces a birational morphism $(\mathfrak{X},\mathcal{H})\rightarrow(\mathfrak{X}',\mathcal{H}')$ where $\mathfrak{X}'=\Proj_\Delta\left(\bigoplus_{n\geq0}\mathcal{O}_\mathfrak{X}(n\mathcal{H})\right)$ (see \cite[Theorem 2, part (i)]{SB83}). This birational morphism contracts some components or curves in the central fiber $\mathfrak{X}_0$.
\end{itemize}

In step $1$, the new Kulikov degeneration is obtained from the one we started by applying elementary modifications of type $0$, I, II (see \cite[pages 12--15]{FM83} for their definitions), base changes, and blow ups of $\mathfrak{X}_0$ along double curves or triple points. The elementary modifications and the blow ups do not change the type of the central fiber. A description of how the central fiber is modified after a base change can be found in \cite{Fri83}, and also in this case the type does not change. In step $2$, it follows from our definition of generalized type that $\mathfrak{X}_0'$ has the same type as $\mathfrak{X}_0$.

This shows that if $\mathfrak{X}_0$ has type I, II, or III, then $\mathfrak{X}_0'$ has type I, II, or III respectively. The converse follows from this and from the uniqueness of the stable model. The claim about the $j$-invariants also follows from our discussion.
\end{proof}


\subsection{Map to the Baily--Borel compactification}

\label{relationksbabailyborel}
Let $L$ be the lattice $\langle2\rangle^{\oplus2}\oplus\langle-1\rangle^{\oplus4}$. By \cite[\S4]{Oud10}, we have that
\[
\mathcal{D}=\{[v]\in\mathbb{P}(L\otimes\mathbb{C})\mid v\cdot\overline{v}>0~\textrm{and}~v^2=0\}
\]
is the period domain for $D_{1,6}$-polarized Enriques surfaces. If $\Gamma$ denotes the isometry group of $L$, then the Baily--Borel compactification $\overline{\mathcal{D}/\Gamma}^*$ was studied in \cite[\S7]{Oud10}. More precisely, the boundary of $\overline{\mathcal{D}/\Gamma}^*$ consists of two rational $1$-cusps and three $0$-cusps. The $1$-cusps are called even and odd. The $0$-cusps are called even, odd of type $1$, and odd of type $2$. These are arranged as shown in Figure~\ref{pictureboundarybailyborel}. In Theorem~\ref{morphismtobailyborelcompactification} we show there exists a birational morphism $\overline{\mathbf{M}}\rightarrow\overline{\mathcal{D}/\Gamma}^*$. Before this we need a preliminary lemma.

\begin{figure}[H]
\centering
\includegraphics[scale=0.35,valign=t]{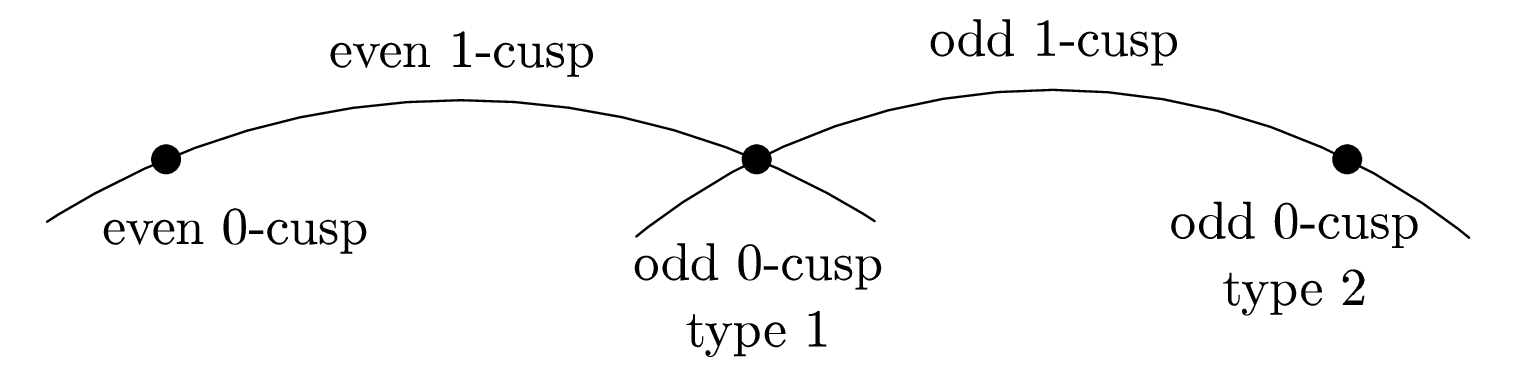}
\caption{Boundary of the Baily--Borel compactification $\overline{\mathcal{D}/\Gamma}^*$.}
\label{pictureboundarybailyborel}
\end{figure}

\begin{lemma}
\label{foldingcompactification}
There exists a compactification $\overline{\mathcal{D}/\Gamma}'$ of $\mathcal{D}/\Gamma$ obtained from $\overline{\mathcal{D}/\Gamma}^*$ by gluing the three $0$-cusps together and by gluing the two $1$-cusps to a rational curve with one node whose smooth points correspond to isomorphism classes of elliptic curves. Moreover, $\overline{\mathcal{D}/\Gamma}^*$ is the normalization of $\overline{\mathcal{D}/\Gamma}'$.
\end{lemma}

\begin{proof}
We use \cite[Theorem 5.4]{Fer03}, which we briefly recall. Let $X'$ be a scheme, $Y'$ a closed subscheme of $X'$, and $Y'\rightarrow Y$ a finite morphism. Consider the ringed space $X=X'\amalg_{Y'}Y$ and the cocartesian square
\begin{center}
\begin{tikzpicture}[>=angle 90]
\matrix(a)[matrix of math nodes,
row sep=2em, column sep=2em,
text height=1.5ex, text depth=0.25ex]
{Y'&Y\\
X'&X.\\};
\path[->] (a-1-1) edge node[above]{}(a-1-2);
\path[->] (a-1-1) edge node[left]{}(a-2-1);
\path[->] (a-2-1) edge node[below]{}(a-2-2);
\path[->] (a-1-2) edge node[right]{}(a-2-2);
\end{tikzpicture}
\end{center}
Let us assume that any finite sets of points in $X'$ (resp. $Y$) are contained in an open affine subset of $X'$ (resp. $Y$). Then $X$ is a scheme verifying the same property on finite sets of points, the above diagram is cartesian, $Y\rightarrow X$ is a closed immersion, the morphism $X'\rightarrow X$ is finite, and it induces an isomorphism $X'\setminus Y'\cong X\setminus Y$.

Back to our case, denote by $C_{\even}$ and $C_{\odd}$ the two rational $1$-cusps of the Baily--Borel compactification $\overline{\mathcal{D}/\Gamma}^*$. Then by \cite[\S7.3]{Oud10} we have that $C_{\even}$ and $C_{\odd}$ are degree $3$ covers of the modular curve $X(1)\cong\mathbb{P}^1$. So first consider the finite morphism $C_{\even}\rightarrow X(1)$ and let $X_1=\overline{\mathcal{D}/\Gamma}^*\amalg_{C_{\even}}X(1)$ (the hypothesis on finite sets of points is satisfied because $\overline{\mathcal{D}/\Gamma}^*$ and $X(1)$ are projective). Repeat the gluing on $X_1$ by considering $C_{\odd}\rightarrow X(1)$ to obtain $X_2$. Now glue together the two copies of $X(1)$ in $X_2$ to obtain $X_3$. Finally, identify the images in $X_3$ of the three $0$-cusps to obtain the claimed compactification $\overline{\mathcal{D}/\Gamma}'$. The isomorphism between $\overline{\mathcal{D}/\Gamma}^*$ and the normalization of $\overline{\mathcal{D}/\Gamma}'$ follows from Zariski's Main Theorem because $\overline{\mathcal{D}/\Gamma}^*$ is normal and the morphism $\overline{\mathcal{D}/\Gamma}^*\rightarrow\overline{\mathcal{D}/\Gamma}'$ is finite and birational.
\end{proof}

\begin{theorem}
\label{morphismtobailyborelcompactification}
There exists a birational morphism $\overline{\mathbf{M}}\rightarrow\overline{\mathcal{D}/\Gamma}^*$ which maps the boundary of $\overline{\mathbf{M}}$ to the boundary of $\overline{\mathcal{D}/\Gamma}^*$.
\end{theorem}

\begin{proof}
The GIT interpretation of $\overline{\mathcal{D}/\Gamma}^*$ in \cite{Oud10} as quotient of the Grassmannian $\Gr(3,6)$ gives a birational map $\overline{\mathcal{D}/\Gamma}^*\dashrightarrow\overline{\mathbf{M}}$ (this morphism is defined at points corresponding to arrangements of six lines in $\mathbb{P}^2$ without triple intersection points). Recall the open subset $\mathbf{A}_{\sm}\subseteq\overline{\mathbf{M}}$ parametrizing stable toric pairs $((\mathbb{P}^1)^3,B)$ with $B$ smooth, and consider the composition $\mathbf{A}_{\sm}\rightarrow\overline{\mathbf{M}}\dashrightarrow\overline{\mathcal{D}/\Gamma}^*$, which is regular. We show that this morphism extends to $\overline{\mathbf{M}}_Q$ giving a $\Sym(Q)$-equivariant morphism. This extension is induced by the universal property of the normalization after we extend to $\overline{\mathbf{M}}_Q$ the composition $\rho\colon\mathbf{A}_{\sm}\rightarrow\overline{\mathcal{D}/\Gamma}^*\rightarrow\overline{\mathcal{D}/\Gamma}'$, where $\overline{\mathcal{D}/\Gamma}'$ was constructed in Lemma~\ref{foldingcompactification}. To extend $\rho$ we use Theorem~\ref{lemmaextensionmapmodulispaces}. So let $K$ be the field of fractions of a DVR $(A,\mathfrak{m})$ and consider any $g\colon\Spec(K)\rightarrow\mathbf{A}_{\sm}$. We show that $\lim(\rho\circ g)\in\overline{\mathcal{D}/\Gamma}'$ can be computed using only $\lim g\in\overline{\mathbf{M}}_Q$.

Let $(X,B)$ be the stable toric pair parametrized by $\lim g\in\overline{\mathbf{M}}_Q$ and consider the corresponding stable pair $\left(B^\bullet,\left(\frac{1+\epsilon}{2}\right)\Delta^\bullet|_{B^\bullet}\right)$. We distinguish the following three cases:
\begin{itemize}
\item Case I: $B^\bullet$ is irreducible;
\item Case II: $B^\bullet$ has exactly two irreducible components glued along an irreducible curve;
\item Case III: otherwise.
\end{itemize}

Denote by $p_0$ (resp. $C_0$) the image of the $0$-cusps (resp. $1$-cusps) under the morphism $\overline{\mathcal{D}/\Gamma}^*\rightarrow\overline{\mathcal{D}/\Gamma}'$. The point $\lim(\rho\circ g)$ can be computed as follows. We have a family of $K3$ surfaces $\mathcal{S}^\mathrm{K3}\rightarrow\mathbf{A}_{\sm}$, which is induced by appropriately quotienting and restricting $\mathcal{X}'\rightarrow\mathbf{U}$ in Definition~\ref{familyD16polarizedEnriquessurfaces}. We have that $\mathcal{S}^\mathrm{K3}\subseteq\mathbf{A}_{\sm}\times(\mathbb{P}^1)^3$, and if $\Delta$ is the toric boundary of $(\mathbb{P}^1)^3$, then $(\mathcal{S}^\mathrm{K3},\epsilon(\mathbf{A}_{\sm}\times\Delta)|_{\mathcal{S}^\mathrm{K3}})$ is a family of stable K3 surface pairs. Now let $\mathcal{Y}'$ be the KSBA completion over $\Spec(A)$ (or a finite ramified base change of it) of the restriction of $\mathcal{S}^\textrm{K3}$ to $\Spec(K)$, where we omitted the divisor for simplicity of notation. In particular, $\mathcal{Y}_\mathfrak{m}'$ is the $\mathbb{Z}_2^3$-cover of $B^\bullet$ branched along $\Delta^\bullet|_{B^\bullet}$ and it is the stable model of a degeneration of smooth K3 surface pairs. Let $\mathcal{Y}$ be a Kulikov degeneration obtained from $\mathcal{Y}'$. Then $\mathcal{Y}_\mathfrak{m}$ determines a unique point in $\overline{\mathcal{D}/\Gamma}'$, which depends on the type of $\mathcal{Y}_\mathfrak{m}$ as follows:
\begin{itemize}
\item Type I: $\lim(\rho\circ g)\in\overline{\mathcal{D}/\Gamma}^*$ is the image under the quotient $\mathcal{D}\rightarrow\mathcal{D}/\Gamma$ of the period point corresponding to $\mathcal{Y}_\mathfrak{m}$;
\item Type II: $\lim(\rho\circ g)\in C_0$ and it corresponds to $j(\mathcal{Y}_\mathfrak{m})$;
\item Type III: $\lim(\rho\circ g)=p_0$.
\end{itemize}

We have that $B^\bullet$ falls into case I (resp. II, III) if and only if $\mathcal{Y}_\mathfrak{m}'$ has generalized type I (resp. II, III). (In particular, by Remark~\ref{explicitdescriptionirreduciblecomponentsdegenerations} $(d1),(d1')$, note that if $B^\bullet$ falls into case I, then $\mathcal{Y}_\mathfrak{m}'$ is smooth or it has isolated singularities of type $A_1$ or $A_3$.) The generalized type of $\mathcal{Y}_\mathfrak{m}'$ equals the type of $\mathcal{Y}_\mathfrak{m}$ by Theorem~\ref{goodbehaviortypeofdegeneration}. In addition, if we are in case II, then it makes sense to define the $j$-invariant $j(\mathcal{Y}_\mathfrak{m}')$, and this equals $j(\mathcal{Y}_\mathfrak{m})$ again by Theorem~\ref{goodbehaviortypeofdegeneration}. In conclusion, we proved that $\lim(\rho\circ g)$ only depends on $\lim g$.
\end{proof}

\begin{remark}
\label{correspondenceoflabelsfor0strata}
Consider the three $0$-cusps of $\overline{\mathcal{D}/\Gamma}^*$ (even, odd of type $1$, and odd of type $2$). Going from left to right in Figure~\ref{explicitstratification}, call the $0$-dimensional strata of $\overline{\mathbf{M}}$ even, odd of type $1$, and odd of type $2$. Let us show that a $0$-dimensional boundary stratum of $\overline{\mathbf{M}}$ maps to the $0$-cusp in $\overline{\mathcal{D}/\Gamma}^*$ with the same label. It is enough to show that a given point in the interior of the maximal $1$-dimensional stratum $\mathbf{C}$ of $\overline{\mathbf{M}}$ is mapped to a point in $C_{\odd}$.

For this purpose, consider a smooth one-parameter family with fibers isomorphic to $\Bl_3\mathbb{P}^2$. Equip this family with a divisor with coefficient $\frac{1+\epsilon}{2}$ cutting on each fiber the usual configuration of three pairs of lines without triple points, but in the central fiber two lines in the same pair come together and the other four lines are general. The limit of this one-parameter family in $\overline{\mathbf{M}}$ lies on $\mathbf{C}$. To show this, one has to compute the stable replacement of the central fiber, and this is done by blowing up the double line and then contracting the strict transforms of the two $(-1)$-curves intersecting the double line. On the other hand, the limit of the same family in $\overline{\mathcal{D}/\Gamma}^*$ was computed by Oudompheng, and it belongs to $C_{\odd}$ (see \cite[Figures~1 and 2]{Oud10}). In light of this and of the proof of Theorem~\ref{morphismtobailyborelcompactification}, it is now clear where a given point of the boundary of $\overline{\mathbf{M}}$ is mapped in $\overline{\mathcal{D}/\Gamma}^*$.
\end{remark}

\begin{remark}
The compactification $\overline{\mathbf{M}}$ is isomorphic to the Baily--Borel compactification over an open neighborhood of the odd $0$-cusp of type $2$. To prove this, observe that the morphism $\overline{\mathbf{M}}\rightarrow\overline{\mathcal{D}/\Gamma}^*$ over this neighborhood is birational and finite, hence an isomorphism by Zariski's Main Theorem (recall that $\overline{\mathcal{D}/\Gamma}^*$ is normal).
\end{remark}


\section{Looijenga's semitoric compactifications}
\label{looijengasemitoriccompactifications}

The behavior of the boundary of $\overline{\mathbf{M}}$ suggests that $\overline{\mathbf{M}}$ could be isomorphic to a semitoric compactification $(\overline{\mathcal{D}/\Gamma})_\Sigma$. These compactifications were introduced in \cite{Loo03}, and they depend on a choice of $\Sigma$, which is called an \emph{admissible decomposition of the conical locus of $\mathcal{D}$}. In this section we briefly recall Looijenga's construction and find a $\Sigma$ such that the boundary of $(\overline{\mathcal{D}/\Gamma})_\Sigma$ is combinatorially equivalent to the boundary of $\overline{\mathbf{M}}$. What follows is a sketch: the full calculations are available in \cite[\S8.6]{Sch17}.


\subsection{Admissible decomposition of the conical locus}
\label{admissibledecompositionoftheconicallocusofd}

Following \cite{Loo03}, let $V=L\otimes_\mathbb{Z}\mathbb{C}$, where the lattice $L=\langle2\rangle^{\oplus2}\oplus\langle-1\rangle^{\oplus4}$ was introduced in \S\ref{relationksbabailyborel}. We call \emph{$\mathbb{Q}$-isotropic} an isotropic subspace $W\subseteq V$ which is defined over $\mathbb{Q}$. Denote by $I$ (resp. $J$) a $\mathbb{Q}$-isotropic line (resp. plane) in $V$. Denote by $C_I\subseteq(I^\perp/I)(\mathbb{R})$ (resp. $C_J\subseteq\bigwedge^2J(\mathbb{R})$) one of the two connected components of $\{x\in(I^\perp/I)(\mathbb{R})\mid x\cdot x>0\}$ (resp. $\bigwedge^2J(\mathbb{R})\setminus\{0\}$). There is a canonical choice for $C_I$ and $C_J$ if we specify a connected component of $\mathcal{D}$, as it is explained in \cite[Sections 1.1 and 1.2]{Loo03}. Denote by $C_{I,+}$ (resp. $C_{J,+}$) the convex hull of the $\mathbb{Q}$-vectors in $\overline{C}_I$ (resp. $\overline{C}_J$). Then the \emph{conical locus} of $\mathcal{D}$ is defined as
\begin{equation*}
C(\mathcal{D})=\coprod_{\substack{W\subseteq V,\\W~\mathbb{Q}\textrm{-isotropic}}}C_W.
\end{equation*}
An \emph{admissible decomposition} of $C(\mathcal{D})$ is a $\Gamma$-invariant locally rational decomposition of $C_{I,+}$ for all $\mathbb{Q}$-isotropic lines $I$, satisfying a certain compatibility condition (see \cite[Definition 6.1]{Loo03}). Such an admissible decomposition provides a compactification of $\mathcal{D}/\Gamma$: for example, the Baily--Borel compactification $\overline{\mathcal{D}/\Gamma}^*$ can be thought of as the semitoric compactification associated to the trivial admissible decomposition of $C(\mathcal{D})$ (see \cite[Example 6.2]{Loo03}). We consider the following admissible decomposition of $C(\mathcal{D})$.

\begin{definition}
\label{ourchoiceofadmissibledecomposition}
For any $\mathbb{Q}$-isotropic line $I$, consider the decomposition of $C_{I,+}$ induced by the mirrors of the reflections with respect to the vectors of square $-1$ in the hyperbolic lattice $(I^\perp/I)(\mathbb{Z})$ (these vectors are called $(-1)$-vectors). Let $\Sigma$ be the decomposition of $C(\mathcal{D})$ induced by these mirrors. It can be shown that this provides an admissible decomposition of $C(\mathcal{D})$ (see \cite[Corollary~8.65]{Sch17}).
\end{definition}

To understand the stratification of $(\overline{\mathcal{D}/\Gamma})_\Sigma$ we need to study the decomposition $\Sigma$. There are three choices of $\mathbb{Q}$-isotropic lines $I\subseteq V$ up to isometries of $L$. In each case, the hyperbolic lattice $(I^\perp/I)(\mathbb{Z})$ is computed in \cite[Proposition 7.5]{Oud10}. For $I$ corresponding to the odd $0$-cusp of type $2$, we have that $(I^\perp/I)(\mathbb{Z})$ is an even lattice. Therefore, the decomposition of $C_{I,+}$ induced by $\Sigma$ is given by $C_{I,+}$ itself. If $I$ corresponds to the even $0$-cusp or the odd $0$-cusp of type $1$, then $(I^\perp/I)(\mathbb{Z})$ are odd lattices. In both cases we run Vinberg's algorithm (see \cite[\S1]{Vin75}) to determine a fundamental domain for $\Gamma_I$, where $\Gamma_I$ is the discrete reflection group generated by the reflections with respect to the $(-1)$-vectors in $(I^\perp/I)(\mathbb{Z})$. The details of the computation can be found in \cite[\S8.6.2]{Sch17}. In summary, we obtain the following.

\begin{theorem}{\cite[Theorem 8.67]{Sch17}}
\label{bijectionstrataksbasemitoriccompactification}
Consider the two birational modifications $\overline{\mathbf{M}}\rightarrow\overline{\mathcal{D}/\Gamma}^*\leftarrow(\overline{\mathcal{D}/\Gamma})_\Sigma$ of the Baily--Borel compactification of $\mathcal{D}/\Gamma$. Then these are isomorphic in a neighborhood of the preimage of the odd $1$-cusp of type $2$. Moreover, there is an intersection-preserving bijection between the boundary strata of $\overline{\mathbf{M}}$ and $(\overline{\mathcal{D}/\Gamma})_\Sigma$, which also preserves the dimensions of the strata.
\end{theorem}



\end{document}